\documentclass[12pt,a4paper]{amsart}

\usepackage{amssymb, amstext, amscd, amsmath, amsfonts, amsthm, amscd, color}

\usepackage[mathscr]{euscript}  
 
\usepackage{tikz,mathdots,enumerate}

\usepackage{graphicx}

\usepackage{tikz-cd}

\usepackage{pb-diagram}

\usepackage{pgfplots}
\usepackage{url}

\newtheorem{thm}{Theorem}[section]

\newtheorem{cor}[thm]{Corollary}

\newtheorem{lemma}[thm]{Lemma}

\newtheorem{prop}[thm]{Proposition}

\numberwithin{equation}{section}

\theoremstyle{definition}
\newtheorem{rem}[thm]{Remark}
\newtheorem{example}[thm]{Example}
\newtheorem{definition}[thm]{Definition}
\newcommand{\bR}{{\mathbb{R}}}

  \newcommand{\B}{{\mathcal{B}}}
  
  \newcommand{\D}{{\mathcal{D}}}

  \newcommand{\M}{{\mathcal{M}}}

\renewcommand{\P}{{\mathcal{P}}}

  \newcommand{\U}{{\mathcal{U}}}

\newcommand{\fP}{{\mathfrak{P}}}

\usepackage{verbatim}

\usepackage{tikz} 
\usetikzlibrary{calc}

\definecolor{green1}{RGB}{58, 156, 58}
\definecolor{green2}{RGB}{51, 255, 51}
\definecolor{green3}{RGB}{153, 255, 153}
\definecolor{green4}{RGB}{204, 229, 204}
\definecolor{green5}{RGB}{102, 127, 102}
\definecolor{green6}{RGB}{37, 178, 25}

\definecolor{red1}{RGB}{255, 153, 160}
\definecolor{red2}{RGB}{178, 25, 36}
\definecolor{red3}{RGB}{204, 127, 132}
\definecolor{purple}{RGB}{171, 25, 178}

\begin{document}

\setcounter{tocdepth}{1}

\title[Projective plane graphs and 3-rigidity]{Projective plane graphs and 3-rigidity}

\author[E. Kastis and S.C. Power]{E. Kastis and S.C. Power}

\thanks{
This work was supported by the Engineering and Physical Sciences Research Council [EP/P01108X/1]}

\address{Dept.\ Math.\ Stats.\\ Lancaster University\\
Lancaster LA1 4YF \\U.K. }

\email{s.power@lancaster.ac.uk}

\email{l.kastis@lancaster.ac.uk}

\maketitle

\begin{abstract}
It is shown that a simple graph which is embeddable in the real projective plane is minimally 3-rigid if and only if it is $(3,6)$-tight. Moreover the topologically uncontractible embedded graphs of this type are constructible from one of 8 embedded graphs by a sequence of vertex splitting moves. In particular the characterisation of minimal 3-rigidity holds for a triangulated  
M\"{o}bius strip.
\end{abstract}


\section{Introduction} 
Let $G$ be the graph of a triangulated sphere. Then an associated bar-joint framework $(G,p)$  in $\bR^3$ is known to be minimally rigid if the placements $p(v)$ of the vertices $v$ is strictly convex (Cauchy \cite{cau}) or if the placement is generic. The latter case follows from 
Gluck's result \cite{glu} that any generic placement is in fact  infinitesimally rigid.
An equivalent formulation of Gluck's theorem asserts that if $G$ is a simple graph which is embeddable in the sphere then $G$ is minimally 3-rigid, in the sense of the next paragraph, if and only if it satisfies a $(3,6)$-tight sparsity condition.
We obtain here the exact analogue of this formulation in the case of simple graphs that are embeddable in the real projective plane $\P$.  
As indicated more fully below, the proof rests on viewing embedded graphs as partial triangulations of $\P$ and employing inductive arguments based on edge contractions for certain admissible edges. Accordingly we may state this combinatorial characterisation in the following form. 
An immediate corollary is that this characterisation  also holds for triangulated M\"{o}bius strips.

A simple graph $G$ is \emph{3-rigid} if its generic bar-joint frameworks in $\bR^3$ are infinitesimally rigid and is \emph{minimally 3-rigid} if no subgraph with the same vertex set has this property.


\begin{thm}\label{t:projectiveA}
Let $G$ be a simple graph 
associated with a  partial triangulation
of the real projective plane. 
Then $G$ is minimally $3$-rigid if and only if $G$ is $(3,6)$-tight.
\end{thm}

 Recall that a graph $G=(V,E)$ is $(3,6)$-tight if it satisfies the Maxwell count $|E|=3|V|-6$ and the sparsity condition $|E'|\leq 3|V'|-6$ for subgraphs $G'$ with an edge or at least 3 vertices. In particular it follows that such a graph falls 3 edges short of arising from a full  triangulation of $\P$.
The proof of Theorem \ref{t:projectiveA} depends heavily on our main result, Theorem \ref{t:construction}, which gives a purely combinatorial constructive characterisation of (3,6)-tight graphs $G$ which are embeddable in the real projective plane.
A key step is a criterion for the existence of an edge contraction move for an  embedded edge  that lies in two 3-cycle faces (called $FF$ edges), such that the $(3,6)$-sparsity condition is preserved. This is done, in Section \ref{s:contractionmoves}, by exploiting the topological environment of the embedded graph. An associated edge contraction sequence must terminate and the terminal embedded graph is said to be \emph{irreducible}. We show that such irreducibles have the defining property that each contractible embedded edge (one for which the contraction is simple) lies on the boundary walk of a $(3,6)$-tight embedded subgraph. These boundary walks are the critical walks in $G$ discussed in Section \ref{ss:criticalcycles}.
 In Section \ref{s:theirreducibles} we show that there are 9 irreducible embedded graphs, including 2 embedded graphs for $K_3$.
As a corollary of this identification we see that the irreducibles coincide with the apparently smaller class of (3,6)-tight embedded graphs which have no contractible edges, that is, they have no $FF$ edges for which the contraction gives a simple graph. 

The identification of the irreducibles makes repeated use of 
Corollary \ref{c:propercontainment}. This ensures that the boundary walk of a nontriangular face of an irreducible cannot be interior, in a natural sense, to a critical cycle of the same length. The various proofs proceed by a case by case analysis according to the existence of $FF$ edges and the number nontriangular faces. In the appendix we give a different proof strategy and determine directly the 9 \emph{uncontractible} embedded graphs. This is a case by case analysis depending on the minimum hole incidence degree given in  Definition \ref{d:holedegree}.


The determination of construction schemes and their base graphs  for various classes of graphs is of general interest, both for embedded graph theory and for the rigidity of bar-joint frameworks.
We note, for example, that Barnette \cite{bar} employed vertex splitting moves for the construction of triangulations of 2-manifolds and showed that there are 2 (full) triangulations of $\P$ which are  uncontractible. Also, Barnette and Edelson \cite{bar-ede-1}, \cite{bar-ede-2} have shown that all 2-manifolds have finitely many minimal uncontractible triangulations. 
With respect to generic rigidity, Fogelsanger \cite{fog} has shown that a finite simple graph given by a triangulated compact surface without boundary is 3-rigid. For the projective plane this was obtained earlier by Whiteley  \cite{whi} using the vertex splitting method and Barnette's characterisation of the uncontractible graphs.
With the exception of the sphere, the graph of a fully triangulated surface without boundary is over-constrained, in the
sense that $|E| > 3|V | - 6$. Characterising the minimal 3-rigidity of partial triangulations is therefore a natural topic and is one which requires additional methods.

Following Whiteley's demonstration that vertex splitting preserves generic rigidity this construction move
has become an important tool in combinatorial rigidity theory \cite{gra-ser-ser}. See, for example, the more recent studies for the graphs of modified spheres \cite{fin-whi}, \cite{fin-ros-whi}, \cite{cru-kit-pow-1},  \cite{jor-tan}, and for the graphs given by a partially triangulated torus \cite{cru-kit-pow-2}. 
The structure of the proof of the main results here follows a similar path to the torus case. In particular we use so-called \emph{face graphs} to define embeddings, as in Figure \ref{f:irred_Five}, where opposite vertices and edges of the boundary of the hexagon are identified.
Also,  Lemmas \ref{l:obstacle1} and \ref{l:obstacle2} are projective plane counterparts of Lemma 4.4 of \cite{cru-kit-pow-2}. On the other hand we find it convenient to introduce \emph{surface graphs}, as defined in Section \ref{s:terminology}, where graphs carry a given triangle face structure.  The proof of Theorem \ref{t:projectiveA}, given in Section \ref{s:mainproof}, follows from Whiteley's theorem, the identification of the irreducible embedded graphs, given in Section \ref{s:theirreducibles}, 
and the fact that the irreducibles have graphs that are minimally 3-rigid.

\section{Surface Graphs}\label{s:terminology} Let $\M$ be a classical surface, by which we shall mean a connected compact surface, possibly with boundary. Then we define 
a \emph{surface  graph for $\M$} to be a triple $G=(V,E,F)$ where $(V,E)$ is a simple graph, with no loop edges, $F$ is a set of $3$-cycles of edges, called facial 3-cycles, and where there exists a faithful embedding of $G$ in $\M$ for which the facial 3-cycles correspond to the 3-sided faces determined by the embedded graph.  A  surface   graph for $\M$, which we also refer to as an \emph{$\M$-graph}, can thus be viewed as a simple graph together with a set of ``facial"  3-cycles which is obtained from a  triangulation of $\M$ by discarding vertices, edges and faces. 
We also say that $G$ is a \emph{triangulated surface graph for $\M$} (or a fully triangulated surface graph for $\M$ for clarity), if no vertices, edges or faces are discarded, so that the union of the embedded faces is equal to $\M$. 

Classical compact surfaces are classified up to homeomorphism by combinatorial surfaces and, moreover,  combinatorial surfaces arise from
triangulated polygons by means of an identification of certain pairs of boundary edges. See \cite{gil-por} for example. We now formally define  labelled graphs of this type {together with their facial structure} and refer to them as \emph{face graphs}. In this definition by a \emph{triangulated disc} we mean, in the terminology above, a triangulated surface graph for the surface which is a closed topological disc.

\begin{definition}\label{d:facegraph}
A \emph{face graph}
is a pair $(B, \lambda)$ where $B$ is a triangulated disc and $\lambda$ is a partition of the edges of the boundary of $B$ such that each set of the partition has $1$ or $2$ edges, and the paired edges of the partition are directed.
\end{definition}

A face graph $(B, \lambda)$  defines a simplicial complex $M$, with $1$-simplexes provided by edges and identified edge pairs, and 2-simplexes provided by the facial 3-cycles, and this complex defines a surface $\M$. We shall be concerned mainly with the projective plane $\P$, and its associated subsurfaces such as M\"obius strips and closed discs and cylinders. We remark that the sphere does not arise in this way but is obtained from a triangulated disc with 3-cycle boundary with this 3-cycle added as an additional facial 3-cycle.

If the identification graph, denoted $B/\lambda$, is simple then it follows that $B/\lambda$, with the set $F$ of facial 3-cycles of $M$, is a triangulated surface graph for $\M$. We also write  $B/\lambda$ for this surface graph.


We now define particular $\M$-graphs in the following similar fashion in terms of a \emph{modified face graph} $(B_0,\lambda)$. By this we mean that 
$B_0$ is a proper subgraph of $B$ which contains the boundary subgraph 
$\partial B$ where $(B,\lambda)$ is a face graph as above. Here $\partial B$ is the graph induced by the edges which have one incident face. Now $(B_0,\lambda)$ gives an identification graph $B_0/\lambda$ with a facial structure inherited from the triangulation of $B$. Also the construction gives a particular embedding of this surface graph in $\M$. Such $\M$-graphs are special in that 
they contain the embedded subgraph associated with $\partial B$. 

\begin{center}
\begin{figure}[ht]
\centering
\includegraphics[width=4.5cm]{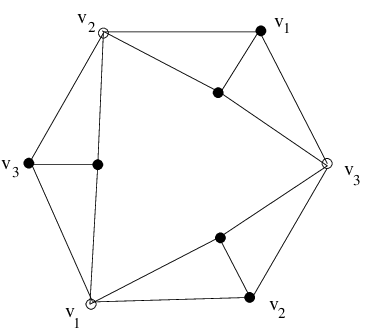}
\caption{A modified face graph $(B_0,\lambda)$ for a $\P$-graph.}
\label{f:mobiussmall}
\end{figure}
\end{center}

{ Let us note that, formally, an \emph{embedding} $\alpha:G\to \M$ of a surface graph $G=(V,E,F)$ in the surface $\M$ is a triple of maps,  $\alpha_V, \alpha_E, \alpha_F$, where $(\alpha_V,\alpha_E)$ is a graph embedding of $(V,E)$ (with $\alpha_E(e)$ a closed set for each edge $e$), and $\alpha_F(f)$ is the closed face of the embedded graph $(\alpha_V(V),\alpha_E(E))$ corresponding to the facial 3-cycle $f \in F$.

\begin{example} \label{e:mobiusfacegraph}
Figure \ref{f:mobiussmall} shows a modified face graph $(B_0,\lambda)$. The labelling of outer boundary edges and vertices determines $\lambda$, that is, the pairs of directed edges that are identified. Any triangulation of the interior of the inner 6-cycle gives a containing face graph $(B,\lambda)$ for a (fully) triangulated surface graph for a surface, as long as $B/\lambda$ is simple. In view of the identifications the topological surface for $(B,\lambda)$ is the real projective plane $\P$ and so $(B_0,\lambda)$ determines a surface graph $G=(V,E,F)=B_0/\lambda$ for $\P$, with 6 facial 3-cycles.
\end{example}

The surface  graph $G$ of Example \ref{e:mobiusfacegraph}  happens to be a fully triangulated surface graph for the M\"{o}bius strip. In general however, the closed set in $\P$ determined by the embedding of a modified face graph and its faces need not be a surface, or surface with boundary. This is because of the possibility of exposed edges which are not incident to any face. This is also true for the embeddings of $(3,6)$-tight graphs that we consider in the next section. See, for example, the first graph in Figure \ref{f:irred_Five}. This exhibits a modified face graph $(B_0,\lambda)$ where $\lambda$ identifies the opposite edges of a 6-cycle and where all other edges and vertices of a containing face graph $(B,\lambda)$ have been removed. It defines a surface graph $G=(V,E,F)$ with $F$ the  empty set and with underlying graph equal to $K_3$.

\subsection{$\P$-graphs}\label{ss:P-graphs} Of particular concern for us are the embeddings of $(3,6)$-tight graphs in the projective plane. 
Let $(B,\lambda)$ be a triangulated face graph for $\P$ which is given by a triangulated disc $B$ whose outer boundary is a directed cycle of even length $2r, r \geq 2$, with $\lambda$ the set of paired opposite directed edges. Let $(B_0,\lambda)$ be a modified face graph for $(B,\lambda)$ such that  
the graph $B_0/\lambda$ is simple. Then, as in the example above, the modified face graph determines a surface graph $H$ for $\P$ and an associated embedding $\pi(H)$. We also write $B_0/\lambda$ and $B/\lambda$ for the surface graphs associated with $(B_0,\lambda)$ and $(B,\lambda)$.

\begin{definition}\label{d:annular} (i) A modified face graph $(B_0,\lambda)$ for $\M$  is  \emph{annular} if it is obtained from a face graph $(B,\lambda)$ by deleting the \emph{internal} edges and vertices of a triangulated subdisc $D$, that is, the edges and vertices not in $\partial D$.

(ii) If $\lambda$ is trivial then $(B_0,\lambda)$ is an \emph{annulus face  graph}, and $B_0$ is a \emph{triangulated annulus}, if the boundary graphs $\partial B$ and  $\partial B_0$ are disjoint. Also  $(B_0,\lambda)$ is a 
\emph{degenerate annulus face graph}, and $B_0$ is a \emph{degenerate triangulated annulus} if the boundary graphs are not disjoint.
\end{definition}

An annular modified face graph $(B_0,\lambda)$ of a $2r$-cycle face graph for $\P$ can be viewed as having a single ``hole". In a similar way we can consider modified face graphs $(B_0,\lambda)$ that contain a number of holes corresponding to the removal of the interior edges and vertices of a number of triangulated disc subgraphs of $(B,\lambda)$ with disjoint interiors.  

We now note that the $\P$-graph for modified $2r$-cycle face graph fails to have the following topological property.

\begin{definition}\label{d:contractible}
An embedding $\pi(H)$ of an $\M$-graph  $H$ in the surface $\M$ is 
\emph{topologically contractible}  if the closed set given by the union of the embedded edges and faces of $H$ is contained in an open subdisc of $\M$. \end{definition}

 To see that the embedded surface graph $\pi(B_0/\lambda)$ in $\P =\pi(B/\lambda)$  is not topologically contractible note first that the image of $\partial B/\lambda$  determines a closed curve in $\P$ with nontrivial homotopy class. On the other hand every closed curve in an open subdisc of $\P$ has trivial homotopy class.

In Lemma \ref{l:converseforuncontractibles} we show that every $(3,6)$-tight embedded $\P$-graph that is topologically uncontractible may be represented by a modified face graph having either 1, 2 or 3 holes.

\begin{rem}  We have found it useful to introduce graphs with an explicit facial structure since the edge contraction operations used in the reduction proofs below are for edges that lie in two facial 3-cycles. Moreover the facial structure arising from an embedding of a simple graph is already given in a modified face graph realisation of the embedding. 

The following equivalent definition of a fully triangulated surface graph is purely combinatorial.  This is of interest since the simplicial complex setting is appropriate for generalisations, both to higher dimensions (with homology cycles generalising surfaces) and for graph embeddings in nonmanifolds.

Let $G = G(M)$ be the 
determined by the $1$-skeleton of a finite simplicial complex $M$
together with a set $F$ of facial 3-cycles determined by the $2$-simplexes
of $M$ where $M$ has the following properties.

\begin{enumerate}[(i)]
\item $M$ consists of a finite set of $2$-simplexes 
together with their $1$-simplexes and $0$-simplexes.
\item Every $1$-simplex lies in at most two $2$-simplexes.
\item The 2-simplexes incident to each 0-simplex induce the simplicial complex of a triangulated disc.
\end{enumerate}

Condition (i)  implies that each 1-simplex lies in at least one 2-simplex and so if $G$ is a connected graph then $M$ can be viewed as a  {combinatorial surface} and this determines a classical topological surface $\M$, possibly with boundary. It follows that $G$, with the facial structure $F$ provided by 2-simplexes, is a triangulated surface graph for  $\M$. 
\end{rem}

\section{Contraction moves and (3,6)-sparsity.}\label{s:contractionmoves}
Let $G=(V,E,F)$ be a surface  graph.
An edge of $G$ is of {\em type $FF$} if it is contained in two facial $3$-cycles and an $FF$ edge is \emph{contractible} if it is not contained in any non-facial $3$-cycle. We say that $G$ is \emph{contractible} if it has a contractible $FF$ edge. For such an edge $e=uv$ there is a natural contraction move $G \to G'$ corresponding to a contraction of $e$ merging $u$ and $v$ to a single vertex, leading to a surface  graph $G'=(V',E',F')$ where $|V'|=|V|-1, |E'|=|E|-3, |F'|=|F|-2$.

To define formally the contracted graph $G'$, let $e=vw$ be a contractible $FF$ edge in $G$ and let $avw$ and $bvw$ be the two facial 3-cycles which contain $e$. Then $G'$ is obtained from $G$ by an \emph{edge contraction} on $e=vw$ if $G'$ is obtained by (i) deleting the edges $aw$ and $bw$, (ii) replacing all remaining edges of the form $xw$ with $xv$, (iii) deleting the edge $e$ and the vertex $w$ and discarding the faces $avw$ and $bvw$.
That $G'$ is simple follows from the fact that a contractible $FF$ edge does not lie on a nonfacial 3-cycle. 

Given an edge contraction move $G \to G'$ we may consider the inverse move, recovering  $G$  from $G'$, which we define to be a \emph{planar vertex splitting move}, or  \emph{vertex splitting move of planar type}, at the vertex $v$. In particular this move introduces a new vertex $w$, 2 new facial 3-cycles, and the new $FF$ edge $vw$. Intuitively, taking account of an embedding of the surface graph $G'$ in a surface $\M$, this corresponds to a construction of a new
surface graph $G$ with embedding in $\M$. For comparison we note the form of a general vertex splitting move defined at the level of simple graphs.

Let $H' = (V',E')$ be a simple graph with vertices $v_1, v_2, \dots , v_r$ and let $v_1v_2, v_1v_3, \dots , v_1v_n$, with $n\geq 3$, be
the edges of $E'$ that are incident to $v_1$. Let $H = (V, E)$ arise from $H'$ by the introduction of a
new vertex $w$, new edges $wv_1, wv_2, wv_3$ and the replacement of any number of the remaining
edges $v_1v_t$, for $t > 3$, by the edges $wv_t$. Then, the move $H' \to H$ is said to be a vertex splitting move on $v_1$.

The \emph{boundary} $\partial G$ of a surface graph $G$ is defined to be the graph determined by the set of edges which are not of $FF$ type. Note that this graph does not depend on any particular embedding of the surface graph.


\subsection{(3,6)-sparse $\P$-graphs.}\label{ss:mobius36tight}
If $H=(V,E)$ is a graph then its \emph{freedom number} is defined to be $f(H)=3|V|-|E|$. Then  $H$ is  \emph{$(3,6)$-sparse} if  $f(H')\geq 6$ for any subgraph $H'$ with at least 3 vertices, or an edge, and is
\emph{$(3,6)$-tight} if it is $(3,6)$-sparse
and $f(H)=6$. In particular a $(3,6)$-tight graph is a simple connected graph, with no loop edges and no parallel edges. We also consider the freedom number of a surface graph to be the freedom number of its underlying graph.

Let us recall that a triangulated surface graph $G=(V,E,F)$ for the sphere $S^2$ is $(3,6)$-tight.  Indeed, regard the associated  graph $H$ as a planar graph and consider a reduction $H\to H'$ on deleting a single interior vertex
and its incident edges and adding chordal edges to the new face to triangulate it. Then $3|V'|-|E'|=3|V|-|E|$ and by induction
$3|V|-|E|$ agrees with the count for the triangle, and so $3|V|-|E|=6$. Similarly for any subgraph $H''$ of $H$ it follows that $3|V''|=|E''|\geq 6$.


\begin{lemma}
Let $G$ be a topologically contractible $(3,6)$-tight surface graph for $\P$ with more than three vertices. Then $G$ has a contractible $FF$ edge.
\end{lemma}

\begin{proof}
It follows from the definition of topologically contractible that $G=(V,E,F)$ has a surface graph embedding in an open subdisc $\D$ of $\P$ and so each face of $G$ maps to a face of the embedding of $(V,E)$ in $\P$ that is contained in $\D$. Since $G$ is $(3,6)$ tight it follows from the previous discussion that the surface graph $G$ is a triangulated surface graph
of a face graph $(B,\lambda)$ where $|\partial B|=3$ and $\lambda$ is trivial. 

The existence of a contractible $FF$ edge clearly holds when k=4. Assume then that such an edge exists whenever  $4 \leq |V|\leq n$ and that $ |V(G)| =n +1$. Consider an interior (non boundary) edge of $G$, say $e =uv$, with associated edges $xu, xv$ and $yu, yv$ for its adjacent faces. If $e$ is not contractible then there is a nonfacial triangle in $G$ with edges $zu, zv, uv$. The subgraph consisting of the 3-cycle $zu, zv, uv$ and its interior determines a new face graph with triangle boundary with fewer vertices than $G$ and it contains at least 4 vertices,
and so by the induction hypothesis $G$ contains a contractible interior edge.  
\end{proof}

Assume now, as in the previous section, that $(B,\lambda)$ is a face graph for $\P$ with $\partial B$ a directed cycle graph of even length $2r, r \geq 2$, with $\lambda$ the set of paired opposite directed edges.
If the identification graph of $B/\lambda$ is simple then $B/\lambda$ is a triangulated surface graph, $S$ say, for $\P$. 
The freedom number $f(B)$ is equal to $6+(2r-3)$, since $B$ may be viewed as a triangulated sphere (which has freedom number $6$) with $2r-3$ edges removed. Noting that $S$ is related to $B$ by the loss of $r$ vertices and $r$ edges it follows that 
\[
f(S)= (3+2r) - 3r + r = 3.
\]  

Let $G$ be a $\P$-graph which is determined by 
an annular modified face graph $(B_0,\lambda)$ associated with $B$. If the  inner boundary cycle has length $s$ then $f(G) = f(S)+ (s-3)$. Thus  $f(G)=6$ if and only if $s=3$. Similarly suppose that $(B_0,\lambda)$
is obtained from $(B,\lambda)$ by removing the interior edges and vertices of several interior-disjoint triangulated subdiscs of $B$. Then $f(G)=6$ if and only if either there are two such subdiscs with boundary cycle lengths 5 and 4, or three subdiscs each with a boundary cycle of length 4.

\begin{lemma}\label{l:tightimplies12or3}
Let $G$ be the $\P$-graph determined by a modified face graph $(B_0,\lambda)$ for the face graph $(B,\lambda)$ for $\P$. If $G$  is (3,6)-tight then $B_0$ has $k$ nontriangular faces, for $k=1, 2$ or $3$ where, for $k=1$ the face has a 6-cycle boundary, for $k=2$ the boundary cycles have length 5 and 4, and for $k=3$ the boundary cycles have length 4.
\end{lemma}

\begin{proof} 
Note first that there can be no interior vertices of $B$ which appear in $B_0$ with degree 1 for these vertices would appear in $G$ with degree 1. It follows that $B_0$ is obtained from $B$ by deleting the interior edges and vertices of several interior-disjoint triangulated subdiscs of $B$. Since the Maxwell count $f(G)=6$ must hold it follows from our previous remarks that $B_0$ satisfies the conditions of the lemma.
\end{proof}

The necessary conditions given in Lemma \ref{l:tightimplies12or3} for $(3,6)$-tightness are not sufficient conditions since, as we see more precisely in Section \ref{ss:criticalcycles}, there are constraints on the lengths of cycles which go around holes.




As we have noted in the previous section, the $\P$-graphs of Lemma \ref{l:tightimplies12or3} are not topologically contractible. On the other hand Lemma \ref{l:converseforuncontractibles} shows that every $(3,6)$-tight $\P$-graph which is not topologically contractible has a modified face graph representation as in Lemma \ref{l:tightimplies12or3}. 

\begin{lemma}\label{l:technicalLemma}
An embedding of a  $(3,6)$-tight surface graph in $\P$ is topologically contractible if the image of each cycle of edges lies in an open subdisc.
\end{lemma}

\begin{proof} A $(3,6)$-tight surface graph $G$ is not a tree and so  contains cycles $c$.
Let $\pi:G\to \P$ be the embedding and note that the open set $\P\backslash \pi(c)$ has two components, one an embedded open disc $\D_c$ with boundary $\pi(c)$, and the other a M\"obius strip. Consider the union of two such open discs together with their boundary curves. This is a proper closed subset since otherwise it would contain a non contractible curve and hence a noncontractible cycle in $\pi(c_1)\cup\pi(c_2)$, contrary to the hypotheses.  Similarly, by induction, the union, $\B$ say, of all the sets $\D_c$ and their boundaries $\pi(c)$ is a proper closed subset which contains $\pi(G)$ (including the images of the faces). Since $G$ is 2-connected the boundary of this closed set is a subset of $\pi(E)$ which is a union of cycles. By construction the boundary must be a single cycle, and $\B$ is an embedded closed disc, and so the lemma follows.
\end{proof}

\begin{lemma}\label{l:converseforuncontractibles}
Let $G$ be a topologically uncontractible  $(3,6)$-tight $\P$-graph. Then there is a modified face graph $(B_0,\lambda)$ for a $2r$-cycle face graph $(B,\lambda)$ for $\P$ such that $G$ is isomorphic to the surface graph $B_0/\lambda$.
\end{lemma}

\begin{proof} Let $\pi$ be an embedding of the face graph $G$ in $\P$. In particular the set of embedded faces, $\pi(F)$, accounts for all the triangular faces determined by $\pi(E)$. 
By the previous lemma there is a cycle $c$ of edges $e_1, \dots , e_r$ for which $\pi(c)$ is not topologically contractible. The complement of $\pi(c)$ in $\P$ is therefore an open subdisc which is partially triangulated by the set by $\pi(F)$. Moreover this partial triangulation may be extended to a full triangulation of $\P$ by triangulating the nontriangular embedded faces. 
This is associated with a corresponding triangulation of a closed disc $B$ with a boundary curve a $2r$-cycle, corresponding to a repetition of the $r$-cycle $c$, and the desired representation of $G$ follows.
\end{proof}

\begin{rem}\label{r:uniqueness}
 We have shown that a topologically uncontractible $(3,6)$-tight $\P$-graph $G$ has a modified face graph representation $(B_0,\lambda)$, with $k=1,2$ or $3$ nontriangular holes, and an associated $2r$-cycle face graph $(B,\lambda)$ for $\P$. This defines a particular surface graph embedding $\alpha:G \to \P$. We now note that any two surface graph embeddings
of $G$ are naturally equivalent. In fact this equivalence is not needed in subsequent arguments.  

Two surface graph embeddings $\alpha, \beta:G \to \P$ are \emph{equivalent}  if 
there exists a homeomorphism $\phi$ of $\P$ such that $\phi\circ \beta = \alpha$, that is, such that the following equalities of closed sets holds:
\[
\phi(\beta_V(v)) = \alpha_V(v),\quad
\phi(\beta_E(e)) = \alpha_E(e),\quad
\phi(\beta_F(f)) = \alpha_F(f), \quad \forall v\in V, e\in E, f\in F.
\]
To see that surface graph embeddings $\alpha, \beta: G\to \P$ are equivalent in this sense we may assume that $\alpha$ is equal to the modified face graph embedding $\lambda_\alpha$ associated with $(B_0^\alpha,\lambda_\alpha)$ and the face graph $(B^\alpha,\lambda_\alpha)$. Thus $\alpha$ has an extension to an embedding $\alpha^+$ where 
\[\alpha: G=B_0^\alpha/\lambda_\alpha \to \P=\P_\alpha, \quad \quad
\alpha^+: G^+= B^\alpha/\lambda_\alpha \to \P=\P_\alpha. 
\]
The second embedding, namely 
\[\beta: G =B_0^\alpha/\lambda_\alpha \to \P_\alpha,
\] 
gives rise to a new partial triangulation of $\P_\alpha$ given by the faces of $\beta(G)$. In other words, with $G=(V,E,F), $ we have the two partial triangulations of $\P$ given by $\alpha(F)$ and $\beta(F)$ as well as a full triangulation $\alpha^+(F^+)$ of $\P$ which extends $\alpha(F)$ by means of  a triangulation  of the remaining nontriangular faces of the embedded graph $\alpha((V,E))$. 
Note that the nontriangular faces of $\alpha((V,E))$ and $\beta((V,E))$ have interiors which are open discs and their boundary walks are given by the embeddings of the boundaries of the holes of $B_0^\alpha$.
In view of this we may construct an extension $\beta^+: G^+ \to \P$ of $\beta$
by triangulating each nontriangular face of $\beta(G)$ with a pattern that matches the given triangulation of the corresponding nontriangular face of $\alpha(G)$.
It is now elementary book-keeping to construct a homeomorphism $\phi$ so that $\phi\circ \beta^+=\alpha^+$. Considering restrictions to $G$ we see that $\alpha$ and $\beta$ are equivalent.
\end{rem}

We make use of following notation for topologically uncontractible  $(3,6)$-tight  $\P$-graphs. 

\begin{definition}\label{d:P_k}
The set $\fP_k,$ for  $k=1,2,3$, is the set of  $(3,6)$-tight  $\P$-graphs which are representable by modified face graphs with $k$ nontriangular faces.
\end{definition}

We also define an \emph{embedded triangulated disc} in $\P$ to be the image of a triangulated disc surface graph under an embedding, 
in the relaxed sense of that distinct faces map to distinct faces but distinct vertices or edges on the boundary may have the same image in $\P$.
Such an embedded triangulated disc can always be extended to a triangulated surface graph for $\P$.
Note that the interior of the closure of the faces of an embedded triangulated disc is evidently homeomorphic to an open disc. On the other hand the closure of the faces need not be homeomorphic to a closed disc and indeed can be equal to $\P$.

\subsection{When contracted surface graphs are $(3,6)$-tight}
A contraction move $G \to G'$ on a contractible $FF$ edge $e$ of a surface  graph preserves the Maxwell count but need not preserve $(3,6)$-tightness. 
We now examine this more closely in the case of a surface  graph for the real projective plane $\P$.

Suppose that $G_1\subseteq G$ and that both $G_1$ and $G$ are in $\fP_1$. 
If $e$ is a contractible $FF$ edge of $G$ which lies on the boundary graph of $G_1$  then, since $G_1$ contains only one of the facial 3-cycles incident to $e$, the contraction $G\to G'$ for $e$ gives a contraction $G'$ which is not $(3,6)$-sparse, since $f(G_1')=5$. We shall show in Lemmas \ref{l:obstacle1}, \ref{l:obstacle2} that the failure of an edge contraction to preserve $(3,6)$-sparsity is due to such  subgraph obstacles. 

The following lemma, which we refer to as the filling in lemma, was  obtained for the torus in Lemma 4.3 of  \cite{cru-kit-pow-2}, and an earlier variant for block and hole graphs is Lemma 26 of \cite{cru-kit-pow-1}. 

\begin{lemma}\label{l:fillingin}
Let $G_*$ be the underlying graph of a $(3,6)$-tight surface graph $G$ for $\P$ and let $H$ be the graph of an embedded triangulated disc graph in $G$ with boundary graph $\partial H$. 

(i) If $K$ is a $(3,6)$-tight subgraph of $G_*$ with $K\cap H = \partial H$ then $\partial H$ is a $3$-cycle graph.

(ii) If $K$ is a $(3,6)$-sparse subgraph of $G_*$ with $f(K)=7$ and $K\cap H = \partial H$ then $\partial H$ is either a $3$-cycle or  $4$-cycle graph.

\end{lemma}

\begin{proof}(i) Write $H^c$ for 
the subgraph of $G_*$ which contains the edges of $\partial H$ and the edges of $G_*$ not contained in $H$.
Since $G_* = H^c\cup H$ and 
$H^c\cap H= \partial H$ we have
\[
6=f(G_*) =f(H^c) +f(H)-f(\partial H).
\]
Since $f(H^c)\geq 6$
we have 
$f(H)-f(\partial H)\leq 0$.
On the other hand,
\[
6\leq f(K\cup H) =
f(K)+ f(H) -f(\partial H)
\]
and $f(K)=6$ and so
it follows  that $f(H) -f(\partial H)=0$.

Let $i:D \to G$ be the triangulated disc embedding with $H$ the underlying graph of the surface graph $i(D)$. Since $i$ is injective on the set of interior vertices of $D$ and the set of interior edges of $D$, it follows that 
\[f(H)-f(D) = f(\partial H) - f(\partial D).\]
We deduce that $f(D)-f(\partial D)=0$. We have $f(D) = 6+(s-3)$ when the boundary is an $s$-cycle, while $f(\partial D) = 2s$ and so $s=3$. It follows that $i(D)$ is a 3-cycle graph (even though, in general, $i$ need not be injective on the boundary edges of $D$).

(ii) The argument above leads to  $-1 \leq f(H) -f(\partial H)$ and hence  $-1 \leq f(D) -f(\partial D)$. It follows now that  
$\partial D$ is either a $3$-cycle graph or a $4$-cycle graph. Since $G_*$ is simple it follows that the graph $\partial H = i(\partial D)$ is also
 a $3$-cycle graph or a $4$-cycle graph. 
\end{proof}


We shall also make use of the following topological property of $\P$.

\begin{lemma}\label{l:topologicalproperty}
Let $(B,\lambda)$ be a face graph for $\P$ and let $U\subset \P$ be a connected open set which is the interior of the union of a set embedded faces of $B$. Then one of the following holds. (i) $U$ is an open disc,  (ii) the closure $\overline{U}$ contains a M\"obius strip, (iii) the complementary  open set $\P\backslash \overline{U}$, is not connected.
\end{lemma}

\begin{proof} Suppose that (ii) does not hold. Then every cycle in $B/\lambda$ is contractible. The proof of Lemma \ref{l:technicalLemma} applies and so $U$ is contained in an embedded open disc. Thus either (i) or (iii) must hold. 
\end{proof}

\begin{lemma}\label{l:obstacle1}
Let $G\in\fP_1$, let $e$ be a contractible $FF$ edge in $G$, and let $G'$ be the simple surface graph arising from the contraction move $G \to G'$ associated with $e$.
Then either $G'\in \fP_1$ or 
$e$ lies on one face of a surface subgraph $G_1$ of $G$, with $G_1\in \fP_1$.
\end{lemma}

\begin{proof} Assume that $G'\notin \fP_1$. It follows that $G'$ must fail the $(3,6)$-sparsity count.
Thus there exists a subgraph $K$ of the underlying graph $G_*$ of $G$ containing $e$ for which the edge contraction results in a graph $K'$ satisfying $f(K')<6$. 
Let $e=vw$ and let $c$ and $d$ be the facial $3$-cycles of $G$ which contain $e$. If  both $c$ and $d$ are subgraphs of $K$ then 
$f(K)=f(K')<6$, which contradicts the sparsity count for $G$. Thus $K$ must contain at most one of these facial $3$-cycles.  

\emph{Case 1}. Suppose first that $K$ contains $c$ and not $d$ and is maximal among all subgraphs of $G_*$ which contain the cycle $c$, do not contain $d$, and for which contraction of $e$ results in a simple graph $K'$ with $f(K')<6$.
Note that $f(K)=f(K')+1$ which implies $f(K)= 6$ and $f(K')=5$.
In particular, $K$ is $(3,6)$-tight, and is a connected graph. Also we may view $K$ as a surface graph for $\P$ endowed with the inherited facial structure from $G$.

Let $(B_0,\lambda)$ be a face graph for $G$ with an associated face graph $(B,\lambda)$ for a triangulated surface graph for $S=(V,E,F)$ for $\P$.
In particular $(B,\lambda)$ provides a faithful topological embedding $\pi: S \to \P$.
Let  $X(K)\subset \P$ be the closed set $\pi_E(E(K))$ and let 
$\tilde{X}(K)$ be the union of $X(K)$ and the embeddings of the faces for the facial $3$-cycles belonging to $K$. Finally,  let $U_1, \dots , U_n$ be the maximal connected open sets of the complement of $\tilde{X}(K)$ in $\P$.

 Note that each such connected open set $U_i$ is determined by a set $\U_i$ of embedded faces of $S$ with the  property: each pair of embedded faces of $U_i$ are the endpoints of a path of edge-sharing embedded faces in $\U_i$. From Lemma \ref{l:topologicalproperty}, $U_i$ has one of the following 3 properties. 
\medskip

(i) $U_i$ is an open disc.

(ii) $U_i$ contains a M\"{o}bius strip.

(iii) The complement of $U_i$ is not connected.   
\medskip

\noindent The third property cannot hold 
since the embedding of $K$  is contained in the complement of $U_i$ and contains the boundary of $U_i$, and yet $K$ is a connected graph.
From the second property it follows that $K$ and its facial 3-cycles is embedded in the complement of a  M\"{o}bius strip and this is an open disc.  This is also a contradiction, since 
the edge contraction of a contractible $FF$ edge in a planar triangulated graph preserves $(3,6)$-sparsity. 



Each set $U_i$ is therefore the interior of the closed set determined by an embedding of a triangulated disc graph in $S$, say $H(U_i)$. Indeed, the facial 3-cycles in $S$ defining $H(U_i)$ are those whose projective plane embedding have interior set contained  in $U_i$.
We may  assume that $U_1$ is the open set that contains the 
single nontriangular face of the embedding of $G$. Thus, if $n=1$ then we may take $G_1$ to be the surface subgraph of $S$ with underlying graph $K$.

Suppose that $i>1$. By the filling in lemma, Lemma \ref{l:fillingin}, it follows  that  $\partial H(U_i)$ is a 3-cycle. 
{Note that no triangulated disc $H(U_i)$ can contain the facial 3-cycle $d$,  for in this case the boundary 3-cycle of $H(U_i)$ contains  $e$ and the contracted graph $K'$ fails to be simple.
By the maximality of $K$ we have $n=1$, since adding the edges and vertices of $S$  interior to these nonfacial 3-cycle boundaries gives a subgraph of $G_*$ with the same freedom count and which does not contain the 3-cycle $d$.} 
Thus, $K$ is the graph of a surface subgraph $G_1$ of $G$ obtained from $S$ by removing the faces of $H(U_1)$ and its interior vertices and edges, and so the proof is complete in this case.
 
\emph{Case 2.}
It remains to consider the case when $K$ contains neither of the facial $3$-cycles $c, d$ which contain $e$. Thus $f(K)=f(K')+2$ and $f(K)$ is 6 or 7.
Once again we assume that $K$ is a maximal subgraph of $G_*$ with respect to these properties and consider the components $U_1, \dots , U_n$ of the complement of $\tilde{X}(K)$.  As before, each set $U_i$ is homeomorphic to a disc and determines an embedded triangulated disc graph $H(U_i)$ in $S$, one of which, say $H(U_1)$, contains the triangulated disc in $S$ associated with the single hole of $G$. If $n=1$ then the proof is complete since we may take $G_1$ to be the surface graph associated with $K$.
On the other hand,
Lemma \ref{l:fillingin} implies that each boundary of $H(U_i)$, for $i>1$, is a $3$-cycle or a $4$-cycle. As in Case 1, maximality implies that a $3$-cycle boundary is not possible. Consider the 4-cycle boundary of  $H(U_i)$, for some $i>1$, and note first that it cannot contain both $3$-cycles $c$ and $d$, since any edge of $K$, and in particular the edge $e$, belongs to $H(U_i)$ only if it belongs to the boundary cycle of $H(U_i)$.
Suppose then that $H(U_i)$ contains $c$ but not $d$. If $f(K)=6$, rather than $7$, then, with $H(U_i)_*$ the underlying graph of $H(U_i)$ we have $f(K\cup H(U_i)_*)=f(K)+f(H_i(U_i))-f(K\cap H(U_i)_*) =6+7-8=5$, contradicting $(3,6)$-sparsity. It follows that $K\cup H(U_i)_*$ is $(3,6)$-tight and contains $c$ but not $d$. Since this is Case 1 the proof is complete.
\end{proof}

The filling in lemma holds for the surface graphs in $\fP_2, \fP_3$ and we may extend Lemma \ref{l:obstacle1} in the following manner.

\begin{lemma}\label{l:obstacle2}
Let $G\in\fP_k$, for $k=1, 2$ or $3$, let $e$ be a contractible $FF$ edge in $G$, and let $G'$ be the simple surface graph arising from the contraction move $G \to G'$ associated with $e$.
Then either $G'\in \fP_k$ or 
$e$ lies on one face of a surface subgraph $G_1$ of $G$, with 
$G_1\in \fP_l$, for some $1\leq l\leq k$.
\end{lemma}

\begin{proof}
The proof for $k=2, 3$ follows the same pattern as in the previous proof for the case $k=1$. Thus we assume that $G' \notin \fP_k$ and consider a subgraph $K$ of $G_*$ subject to two cases. In Case 1 $K$ is maximal among all subgraphs which contain $c$ and not $d$, where $c, d$ are the 3-cycles incident to $e$, and $f(K')=5$. In Case 2 $K$ is maximal among subgraphs of $G$ not containing $c, d$ and for which $f(K')$ is equal to 4 or 5. We consider again the open set which is the complement of the embedding in $\P$ of $K$ and its facial 3-cycles. 
This open set has connected components $U_1, \dots , U_n$ and each is the interior of a union of an edge-connected set of $\P$-embedded facial 3-cycles of the surface graph $S$ for $\P$. Also the graphs $H(U_j)$ are the associated surface subgraphs of $S$. It follows as before that each $U_j$ is an open disc. 

Suppose that $H(U_j)$ does not contain any of the $k=2$ or $3$ triangulated discs which define $G$. In Case 1, by the filling in lemma the boundary of $H(U_j)$ must be a 3-cycle,  and so the proof is completed as before. Indeed, $1\leq n \leq 3$ and each of the surface graphs $H(U_1), \dots , H(U_n)$ contains at least one of the triangulated discs that define $G$ and we may take $G_1$ to be the surface graph determined by $K$.
In Case 2, by the filling in lemma, the boundary of  $H(U_j)$ is a 3-cycle or a 4-cycle. As before, by maximality, the boundary is not a 3-cycle. 
If it is a 4-cycle then either $K\cup H(U_j)$ contradicts the maximality or $H(U_j)$ or $K\cup H(U_j)$ contains one or both of $c, d$.
As in the previous proof, containing both is not possible and so $K\cup H(U_j)$ contains one of $c$ and $d$. Since this is Case 1, the proof is complete.
\end{proof}

We remark that Lemma \ref{l:obstacle2} is analogous to the critical cycle lemma, Lemma 27 of \cite{cru-kit-pow-1}, for $(3,6)$-tight block and hole graphs where, roughly speaking, there is a single block that is  complementary to face graph of a multiconnected surface graph in the plane.

\subsection{Critical embedded cycles}\label{ss:criticalcycles}
Lemma \ref{l:obstacle2} reveals the obstacle to the preservation of $(3,6)$-sparsity when contracting the contractible edge $e$ of a surface graph $G$ in $\fP_k$, namely that $e$ lies on the boundary of a surface subgraph $G_1$ of $G$ which is in $\fP_l$ for some $l\leq k$. 
Let $S$ be a triangulated surface graph for $\P$ that contains $G$, so that $G$ is given by removing the interior vertices, edges and faces of interior disjoint embedded triangulated discs $\pi(D_1), \dots ,\pi(D_k)$. Then $G_1$ is given similarly in terms of interior-disjoint 
embedded triangulated discs $\mu_1(B_1), \dots , \mu_l(B_l)$, where the edge $e$ is in $\mu_1(\partial B_1)$ and $\mu_1(B_1)$ contains one or more of  
$\pi(D_1), \dots , \pi(D_k)$. 
If $c_1$ is the boundary $r$-cycle of $B_1$ then we refer to $c=\mu_1(c_1)$ as an \emph{critical embedded $r$-cycle} or  \emph{critical walk (or $r$-walk)} in $G$, where $|\partial B_1|= r$, with $r=4,5 $ or $6$. We also use this terminology to include the boundary walks $\pi(D_i)$ of the surface graph $G$. A boundary walk is formally a sequence of edges and for simple graphs can  be indicated by specifying the associated sequence of vertices. 

Note that $G_1$ is a vertex-induced surface subgraph of $G$ since it is $(3,6)$-tight.  
Also, considering Maxwell counts, it follows in all cases that 
\[
r-3= |\partial B_1|-3 = \sum_{i\in I}(|\partial D_i|-3) 
\]
where $I=\{i:\pi(D_i)\subseteq \mu_1(B_1)\}$.

An \emph{embedded planar $r$-cycle} in a surface graph $G$ is defined to be the boundary walk of an embedded triangulated disc in $G$.
In particular a critical embedded $r$-cycle of $G\in \fP_k$, for some $r=4,5$ or 6, is not an embedded planar $r$-cycle in $G$. 

\begin{lemma}\label{l:planar3cycle}
Let $G$ be a surface graph in $\fP_k$ for some $k=1,2,3$, with an embedded planar 3-cycle which is not a facial 3-cycle. Then there is a contractible edge $e$ with $G/e$ in $\fP_k$.
\end{lemma}

\begin{proof}
Let $\pi(\partial B)$ be the planar embedded 3-cycle with triangulated disc $B$ with 4 or more vertices. Then there is an edge $e'$ of $B$, not in $\partial B$, for which $B/e'$ is simple and hence the edge $e=\pi(e') $ is a contractible edge. Since a  surface subgraph associated with  a critical walk is an induced subgraph it follows that $e$ cannot lie on a critical walk, and so $G/e$ is in $\fP_k$.
\end{proof}

It is convenient to refer to a 3-cycle in $G$ as an \emph{essential 3-cycle of $G$} if it is not a facial 3-cycle and not an embedded planar 3-cycle. In fact these cycles are the 3-cycles with an associated homotopy class that is nonzero.


The following useful lemma enables the creation of new critical embedded cycles by means of subwalk substitutions. Figure \ref{f:blendingCriticals}(i) gives an intuitive sketch of this. The bold ellipses indicate critical walks of lengths $4$ and $5$ in a surface graph $G$ in $\fP_3$ and the boundary walk of the union of the interior of the ellipses is necessarily a critical walk.


\begin{lemma}
\label{l:blendingcriticals} 
Let $G$ be a surface graph in $ \fP_k$, with $1\leq k\leq 3,$ and let $c_a, c_b$ be critical walks in  $G$ with associated embedded discs $\pi(D_a), \pi(D_b)$ and surface subgraphs $G_a, G_b$ in $ \fP_l$, for $1\leq l\leq k,$. Suppose that  $G_a\cap G_b$ contains a face and the embedded open discs $\pi(D_a^\circ),  \pi(D_b^\circ)$ have union equal to an embedded open disc. 
Then the boundary walk of $\pi(D_a)\cup \pi(D_b)$ is a critical  walk of the (3,6)-tight surface graph $G_a\cap G_b$.
\end{lemma}

\begin{proof}
We have the freedom count equation
\[
f(G_a\cup G_b)=f(G_a)+f(G_b)-f(G_a\cap G_b)
\]
together with $f(G_a)=6, f(G_b)=6$ and  $f(G_a\cap G_b)\geq 6$.
Thus $f(G_a\cup G_b)=6$ in view of (3,6)-sparsity. Thus $f(G_a\cap G_b)=6$ and so $G_a\cup G_b$ and $G_a\cap G_b$ are $(3,6)$-tight.
The hypotheses ensure that one of the embedded discs that defines $G_a\cap G_b$ is equal to  $\pi(D_a\cup D_b)$ and so $\pi(D_a\cup D_b)$ has a well-defined boundary walk. It follows then that this walk is  a critical embedded cycle in $G$. 
\end{proof}

\begin{center}
\begin{figure}[ht]
\centering
\includegraphics[width=4.5cm]{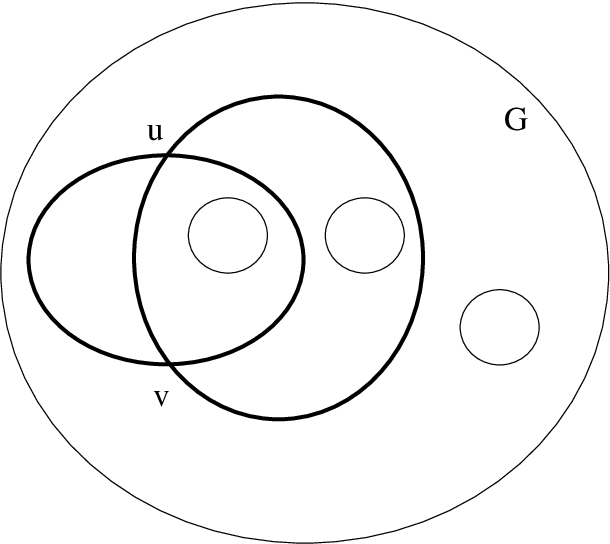}\quad \quad 
\includegraphics[width=4.5cm]{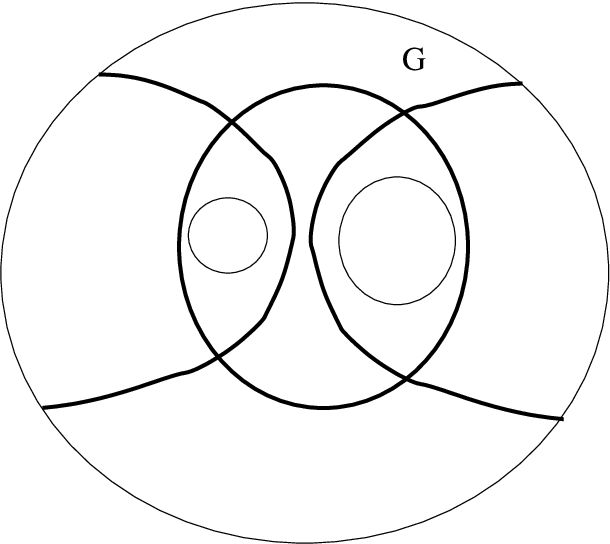}
\caption{Schematic diagrams of embedded critical walks (in bold) in $\P$. The union of the interiors, in $\P$, of the ellipses is homeomorphic to (i) an open disc, (ii) an open M\"{o}bius strip.} 
\label{f:blendingCriticals}
\end{figure}
\end{center}
\begin{center}
\begin{figure}[ht]
\centering
\includegraphics[width=4cm]{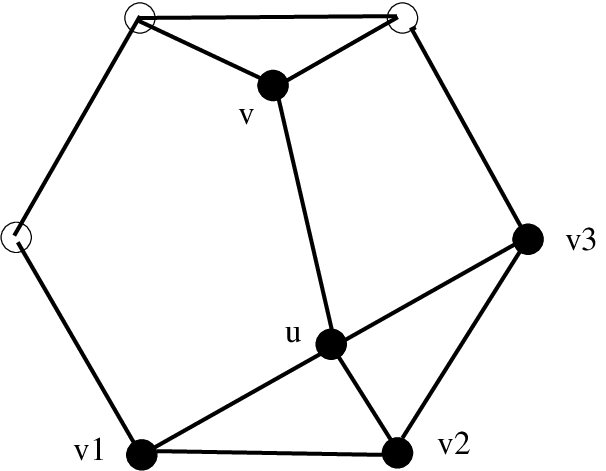}\quad \quad 
\caption{The critical embedded 6-cycle $v_1uv_3v_2v_2$ is not the repetition of an essential 3-cycle.} 
\label{f:2criticalsFailingHypothesis}
\end{figure}
\end{center}


Figure \ref{f:2criticalsFailingHypothesis} gives the modified face graph
$(B_0,\lambda)$ of a surface graph $G$ in $\P_2$ which  illustrates various of critical 6-walks. The closed walk $v_1v_2v_3v_1v_2v_3v_1$ is the repetition of the essential 3-cycle $v_1v_2v_3$ and its associated embedded disc in a containing triangulated surface graph $B/\lambda$ for $\P$ contains all the faces of $\pi(B)$. Other critical walks of this type are associated with the essential 3-cycles $v_1uv_3$, $v_1uv$, $v_2uv$. On the other hand the critical 6-walks $v_1uv_3v_1v_2$ and
$v_1uvv_2uv$ are not repeated 3-cycles.

\section{The irreducibles}\label{s:theirreducibles}

Let $k=1,2$ or $3$. Then a surface graph $G$ in $\fP_k$  admits a reduction sequence
\[
G=G_1 \to G_2 \to \dots \to G_n
\]
where (i) each $G_k$ is in  $\fP_k$, (ii) each move $G_k \to G_{k+1}$ is an edge contraction for an $FF$ edge,  and (iii) $G_n$ is \emph{irreducible} in  $\fP_k$ in the sense that it admits no edge contraction to a surface graph in  $\fP_k$. We show in this section that there are 8 such irreducible surface graphs, denoted  $G_{b,1}, \dots , G_{b,8}$ and also referred to as \emph{base graphs}. They are given in Figure \ref{f:irred_Five} in terms of modified face graphs. 
\begin{center}
\begin{figure}[ht]
\centering
\includegraphics[width=2cm]{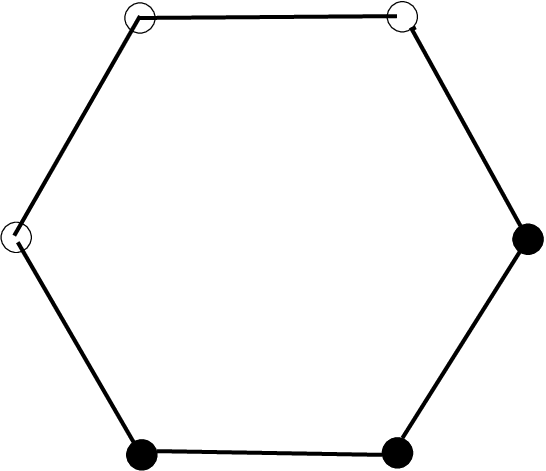}\quad 
\includegraphics[width=2cm]{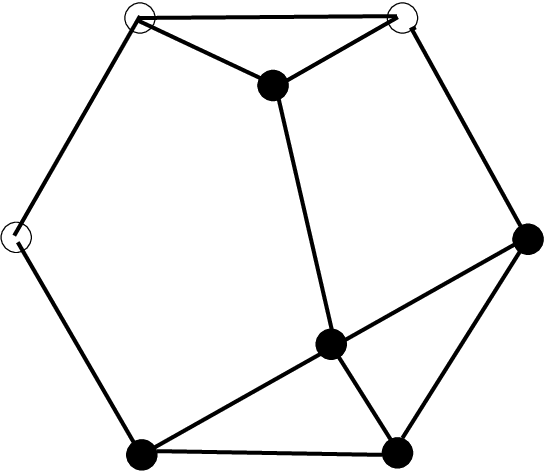}\quad 
\includegraphics[width=2cm]{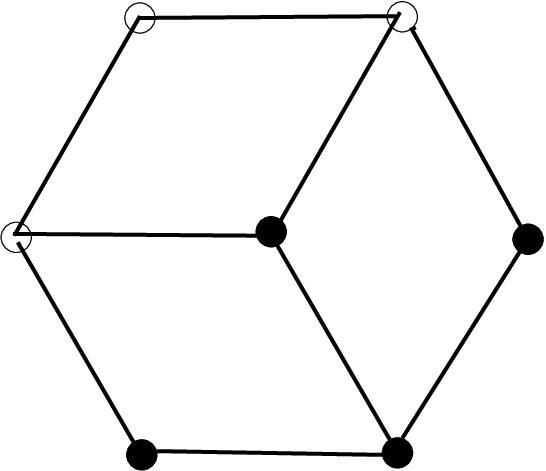}\quad 
\includegraphics[width=2cm]{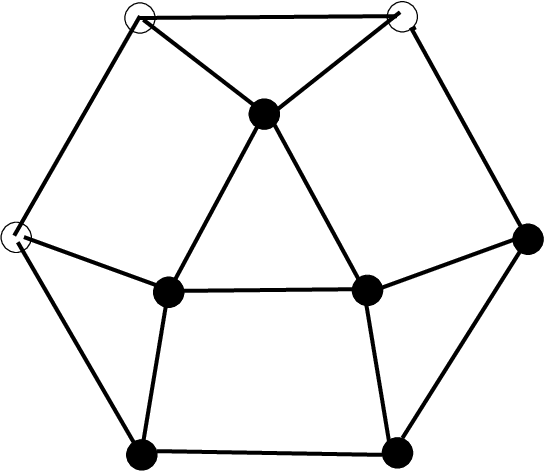}\quad
\end{figure}
\end{center}
\begin{center}
\begin{figure}[ht]
\centering

\includegraphics[width=2cm]{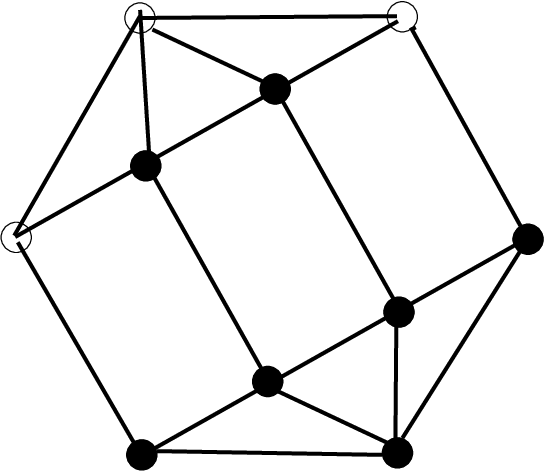}\quad
\includegraphics[width=2cm]{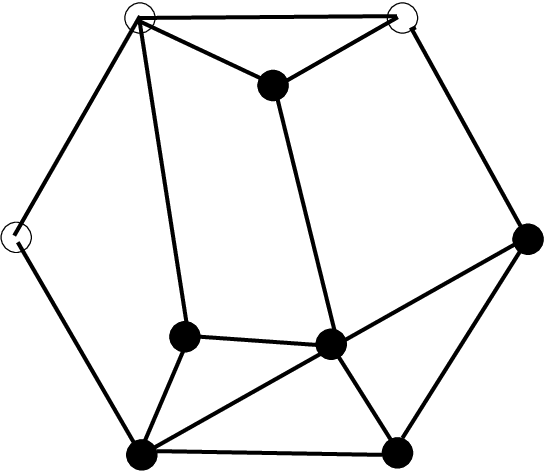}\quad
\includegraphics[width=2cm]{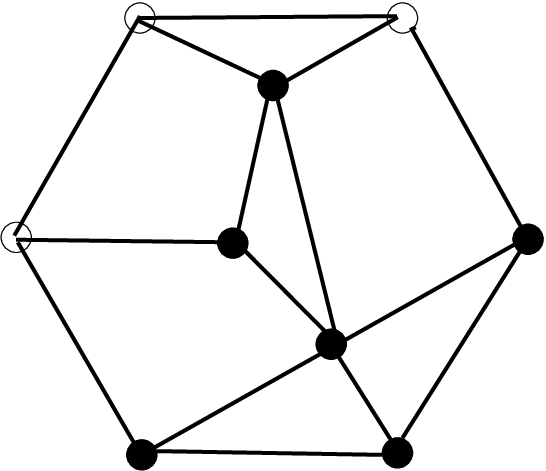}\quad 
\includegraphics[width=2cm]{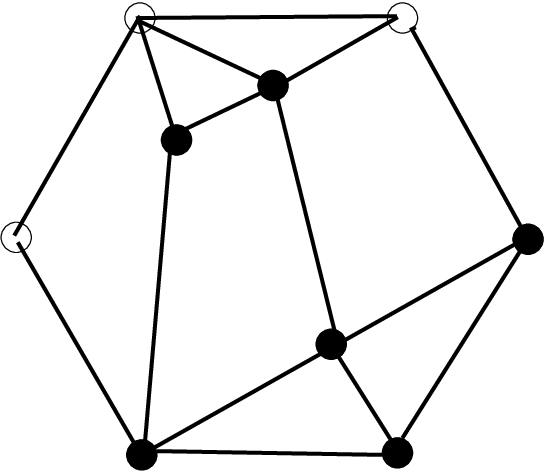}\quad 
\caption{Modified face graphs for the irreducibles $G_{b,1}, \dots , G_{b,8}$.} 
\label{f:irred_Five}
\end{figure}
\end{center}
\begin{center}
\begin{figure}[ht]
\centering 
\includegraphics[width=2cm]{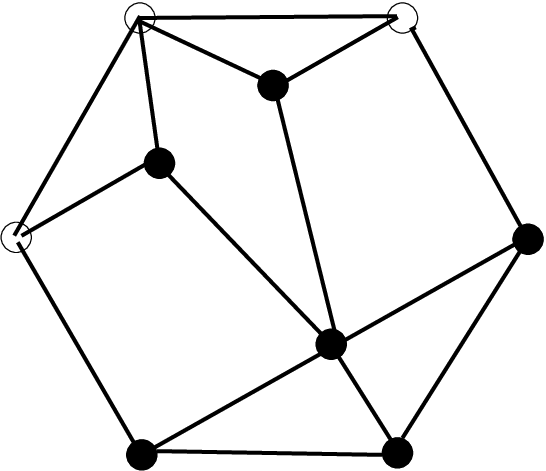}\quad \quad
\includegraphics[width=2cm]{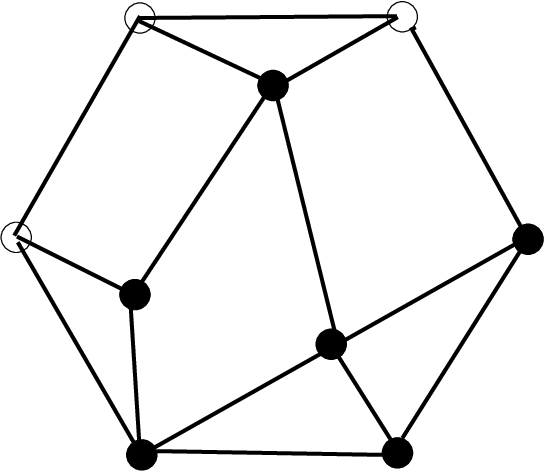}
\caption{Isomorphs of  (i) $G_{b,8}$, (ii) $G_{b,7}$.} 
\label{f:irred_PlusExtra}
\end{figure}
\end{center}

A  surface  graph is \emph{uncontractible} if every $FF$ edge lies on a nonfacial 3-cycle and so
an uncontractible surface graph $G$ in $\fP_k$ is certainly irreducible. That the reverse implication also holds is a corollary of the identification of the irreducibles.

Let us also note that the surface graphs $G_{b,6}, G_{b,7}, G_{b,8}$ are augmentations of $G_{b,2}$ by a degree 3 vertex and a facial 3-cycle.  They are nonisomorphic surface graphs by virtue of the fact that they are nonisomorphic as graphs. Figure \ref{f:irred_PlusExtra} shows other representations of  $G_{b,7}$ and $G_{b,8}$ which arise this way.
Also, as graphs we have $G_{b,1}=K_3, G_{b,2}= K_5-e, G_{b,3}=K_4$ and  $G_{b,5}$ is the cone over $K_{3,3}$. The remaining graphs $G_{b,4}, G_{b,6},G_{b,7},G_{b,8}$ are depletions of $K_6$ by  3 edges where these edges 
are (i) disjoint, (ii) form a copy of $K_3$, (iii) have 1 vertex incident to a pair of the edges,  (iv) have 2 vertices incident to a pair of the edges. 
In fact the 8 graphs account for all possible $(3,6)$-tight graphs on $n$ vertices for $n=3, 4, 5, 6$, together with 1 of the 26 such graphs for $n=7$. We remark that for $n=8, 9, 10$ the number of $(3,6)$-tight graphs rises steeply, with values 375, 11495, 613092 (Graseggar \cite{gra}).


\begin{lemma}\label{l:degree3lemma} 
Let $G$ be a surface graph in $\fP_k$, with $k=1, 2$ or $3$.
If $e$ is an $FF$ edge that is incident to a degree 3 vertex then
the contraction $G/e$ is in $\fP_k$.
\end{lemma}

\begin{proof} 
The edge $e$ is contractible, that is, $G/e$ is simple, since it does not lie on a nonfacial 3-cycle. Let $e$ be the edge $uv$ with $\deg(v)=3$ and with facial 3-cycles $uvx$ and $uvy$. If it lies on a critical embedded 4-, 5- or 6-cycle, then the associated surface subgraph fails to be vertex induced. Thus Lemma \ref{l:obstacle2} completes the proof.
\end{proof}

\begin{lemma}\label{l:oneholelemma}
 Let $G_1 \subset G$ be  surface graphs in $\fP_1$ determined by embeddings  $\pi(D_1)$ and $\pi(D)$, respectively, of triangulated discs, where $D$ is a proper subset of $D_1$.   
Then $G$ is constructible from $G_1$ by planar vertex splitting moves. 
\end{lemma}

\begin{proof}
Suppose that $|V(G)|=|V(G_1)|+1.$ 
It follows that the vertex of the boundary of $G$ which is not in $G_1$ has degree $3$. By Lemma \ref{l:degree3lemma} $G$ is constructible from $G_1$ by a single planar vertex splitting move. 

Assume next that the lemma is true whenever $|V(G)|=|V(G_1)|+j$, for $j = 1,2,\dots , N-1$, and suppose that $|V(G)|=|V(G_1)|+N$. We claim that the embedded annulus graph $\pi(A)$, where $A$ is defined by removing the interior vertices and edges of $D$ from $D_1$, has an $FF$ edge. To see this let $v_1=\pi(w_1)$ be a vertex of the embedded 6-cycle $c=\pi(c_1)$, where $c_1=w_1w_2\dots w_6$ is the boundary cycle for $D_1$, which is not a vertex of $D$. Since, by $(3,6)$-sparsity, $G_1$ is an induced graph in $G$ there is no edge $\pi(w_2)\pi(w_6)$ in $G_1$. It follows that there is an interior edge $e=v_1v$ in $\pi(A)$ and moreover that $e$ is an $FF$ edge, since $A$ is a (possibly degenerate) triangulated annulus.

If the contraction $G/e$ is in $\fP_1$ then it follows from the induction step assumption that $G$ is constructible from $G_1$ by planar vertex splitting moves. So, by Lemma \ref{l:obstacle1} we may assume  (i), that $e$ lies on a critical embedded $6$-cycle, $c_e$ say, 
or  (ii), that $e$ lies on a nonfacial 3-cycle of $G$. In the former case
we may substitute a subwalk of the critical embedded 6-cycle $c_b=\pi(\partial D_1)$ by a subwalk of $c_e$ of the same length which is interior to $\pi(\partial D_1)$ and has the same initial and final vertices and where these subwalks form the boundary of an embedded triangulated disc in $\pi(D_1\backslash D)$.
The resulting critical embedded 6-cycle $c'$ lies strictly inside $c_b$. If $G_1'\in \fP_1$ is its associated surface graph then we have
  $|V(G)|-|V(G_1')|<N$ and  $|V(G_1')|-|V(G_1)|<N$, and it follows from the  induction step that the lemma holds for $G$ and $G_1$. 
 
We may assume then that (ii) holds. 
Since $e$ is an $FF$ edge in $\pi(A)$ it follows that any nonfacial 3-cycle containing $e$ is embedded planar 3-cycle in $G$ of the form $\pi(\partial B)$ for an embedded disc $\pi(B)$ with $B$ contained in $A$. By Lemma \ref{l:planar3cycle} there is a contractible edge $f$ with $G/f$ in $\fP_1$ and so the induction step follows in this case also.
\end{proof}

The previous lemma shows that $G_{b,1}$ is the unique irreducible surface graph in $\fP_1$. It is the surface graph given by the graph $K_3$ together with an empty facial structure. 
For the surface graphs $G$ in $\fP_2, \fP_3$ the reduction arguments are more involved since, the embedded disc associated with a critical $s$-walk, for $s=5$ or $6$ can contain more than one boundary walk of $G$. However, the next lemma  implies that $G$ is not irreducible if there is a critical walk that properly contains a critical walk of the same length, and we use this corollary frequently.

\begin{lemma}\label{l:mholelemma}
 Let $r=2$ or 3 and let $G_1 \subset G$ be surface graphs in $\fP_r$ determined by embedded triangulated discs  $\pi(B_i), 1\leq i\leq r,$ and $\pi(D_i), 1\leq i\leq r,$, respectively, where $D_i$ is a proper subset of $B_i$ for each $i$.   
Then $G$ is constructible from $G_1$ by planar vertex splitting moves. 
\end{lemma}

\begin{proof} We may argue by induction as in the previous lemma.
Suppose for example that $r=2$ and that $e$ is an $FF$ edge of $G\backslash G_1$ and that $e$ is the embedding of an edge in $B_1\backslash D_1$ where the boundary cycles of $B_1, D_1$ are of length 5. 
The induction step can be completed if $G/e$ is in $\fP_2$ and so we may assume that either (i), $e$ is on a critical embedded $s$-cycle $c_e$, for $s=5$ or 6, or (ii),  $e$ lies on a nonfacial 3-cycle. If (i) holds with $r=6$ then we may apply Lemma \ref{l:blendingcriticals} to see that $e$ also lies on a critical embedded 5-cycle $c_e'$ which contains the embedding of $B_1$, in the usual sense.
Thus $c_e'$ determines an intermediate surface subgraph in $\fP_2$, and so we may complete the induction step as in the previous proof. If (i) holds then, as before, we similarly obtain an intermediate surface subgraph in $\fP_2$. For the case $r=3$ Lemma \ref{l:blendingcriticals} can again be used to obtain an intermediate surface graph in $\fP_3$ and so in all cases the induction step may be completed.
\end{proof}

\begin{cor}\label{c:propercontainment}
 Let $r=2$ or 3 and let $G_1 \subset G$ be  determined by embedded triangulated discs  $\pi(B_i), 1\leq i\leq r,$ and $\pi(D_i), 1\leq i\leq r,$, respectively. If $D_1$ is a proper subset of $B_1$ and $D_i=B_i$ for $i\neq 1$ then $G$ is not irreducible. 
In particular, if $G$ is irreducible then it is not possible for $G$ to have a critical embedded 4-cycle containing an $FF$ edge.
\end{cor}

We next determine the surface graphs which contain no $FF$ edges. From the definition these are necessarily irreducible. The next three propositions, together with our remarks following Lemma \ref{l:obstacle1}, show that they are $G_{b,1}$ in $\fP_1$ and $G_{b,3}, G_{b,4}$ in $\fP_3$.

\begin{prop}\label{p:noFacesB}
Let $G$ be a surface graph in $\fP_k$, for $k=2$ or 3. If $G$ has  no facial 3-cycles then $k=3$ and $G = G_{b,3}$.
\end{prop}

\begin{proof} Let $(B_0,\lambda)$ be a modified face graph representation for $G$ with boundary a $2r$-cycle. Note that for $k=2$ or 3 each vertex of $G$ that is not on the embedded $2r$-cycle has degree 3. In particular removing this vertex and its incident edges does not change the Maxwell count of 6. It follows, by removing all interior vertices that $r$ must be equal to 3 and the conclusion follows.
\end{proof}

\begin{prop}\label{p:noFacesB}
Every surface graph in $\fP_2$ has $FF$ edges.
\end{prop}

\begin{proof} Let $G$ in $\fP_2$ have a modified face graph representation $(B_0,\lambda)$ with outer boundary cycle of length $2r$. We may assume from the previous proposition that $G$ has faces. Suppose, by way of contradiction, that $B_0$ has 2 holes and $G$ has no $FF$ edges. Then there is an edge $xy$ in the boundary $2r$-cycle of $B_0$ which has a face $xyv$. The paired edge $x'y'$ therefore belongs to the boundary cycle of one of the holes of $B_0$. See Figure \ref{f:2holesNoFF} where $D_1$ indicates this hole. Note that it is not possible for $D_1$ to be incident to $x$ or $y$. Indeed, if $D_1$ has a 5-cycle boundary then since the edges $xy'$ does not exist, by the simplicity of $G$, the boundary walk of $D_1$ from $x$ to $x'$ has length 1 or 2, contradicting the simplicity of $G$. Since there are no $FF$ edges the other hole boundary must contain the edges $xv$ and $yv$, and this is a contradiction since it implies that $v$ has degree 2.
\begin{center}
\begin{figure}[ht]
\centering
\includegraphics[width=3.5cm]{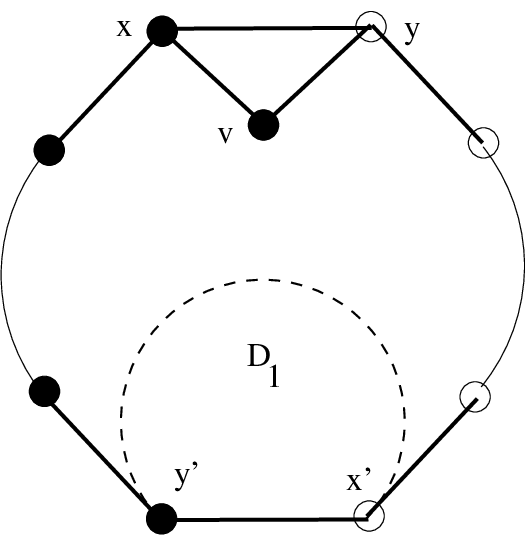} 
\caption{$(B_0,\lambda)$ with a face and no $FF$ edges.} 
\label{f:2holesNoFF}
\end{figure}
\end{center}
\end{proof}

\begin{prop}\label{p:faceNoFF}
Let $G$ be a surface graph in $\fP_3$ with facial 3-cycles and no $FF$ edges. Then $G = G_{b,4}$.
\end{prop}

\begin{proof}
Let $(B_0,\lambda)$ be a modified face graph representation for $G$ with boundary a $2r$-cycle. We show that $r$ can be $3, 4$ or 5 and in all cases the surface graph $G= B_0/\lambda$ is isomorphic to $G_{b,4}$.

\begin{center}
\begin{figure}[ht]
\centering
\includegraphics[width=4cm]{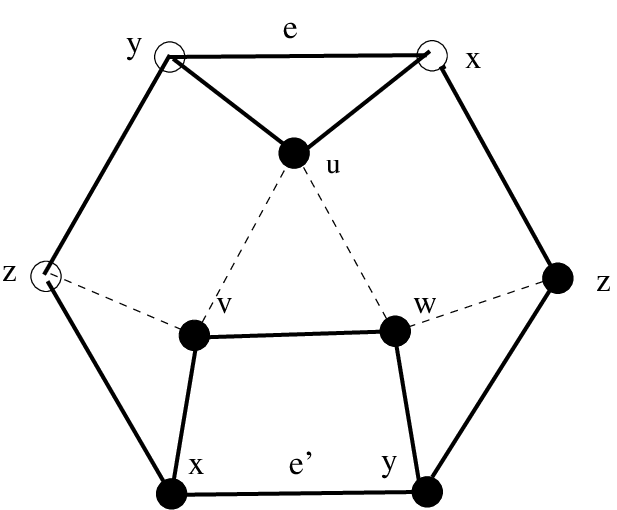} \quad \quad
\includegraphics[width=4cm]{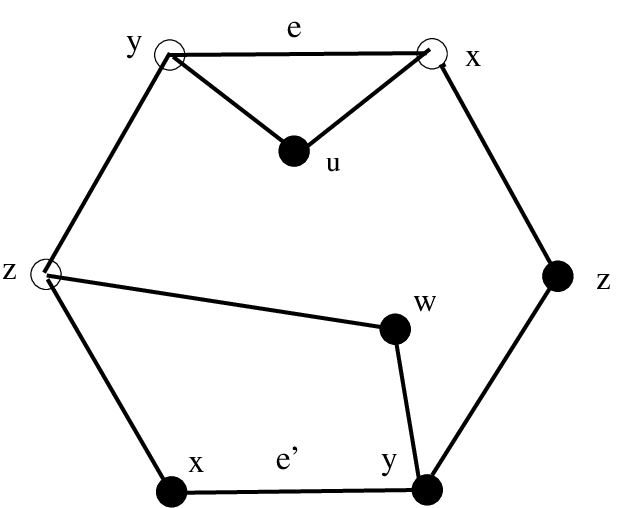}
\caption{Possible subgraphs of $(B_0,\lambda)$ when $r=3$ (shown with bold edges).} 
\label{f:case_ris3}
\end{figure}
\end{center}
We may assume that $B_0$ has a boundary edge $e=xy$ on a facial 3-cycle. Let $e'$ be the edge in the boundary of $B_0$ which is identified with $e$ by $\lambda$. Suppose first that $r=3$. By our assumptions $e'$ belongs to the boundary 4-cycle of a hole of $B_0$. It follows that $B_0$ contains a surface subgraph with one of the forms given in Figure \ref{f:case_ris3}.

In the first case the boundary 4-cycle is $v,w,y,x$ with vertices $v, w \neq z$.
The remaining two boundary 4-cycles of $B_0$ contain the edges $uy$ and $ux$ respectively, and neither 4-cycle contains both edges. Since there are no $FF$ edges it follows that $G=G_{b,4}$. In the second case, with  $z=v$, we have $u \neq w$, by the simplicity of $G$. Since there are no $FF$ edges it follows that $ywz$ is a face since otherwise the boundary 4-cycle containing $ux$ must be incident to $y$, contradicting simplicity. But the presence of the face $ywz$ also  contradicts the simplicity of $G$, and so the second case does not arise.
\begin{center}
\begin{figure}[ht]
\centering
\includegraphics[width=3.5cm]{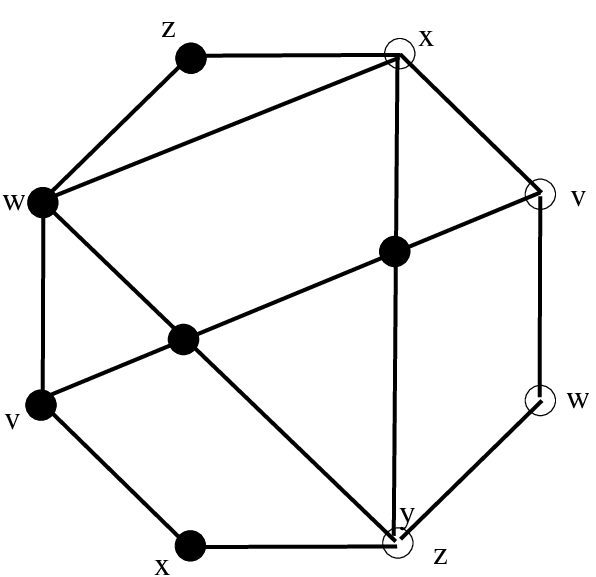} \quad \quad
\includegraphics[width=4cm]{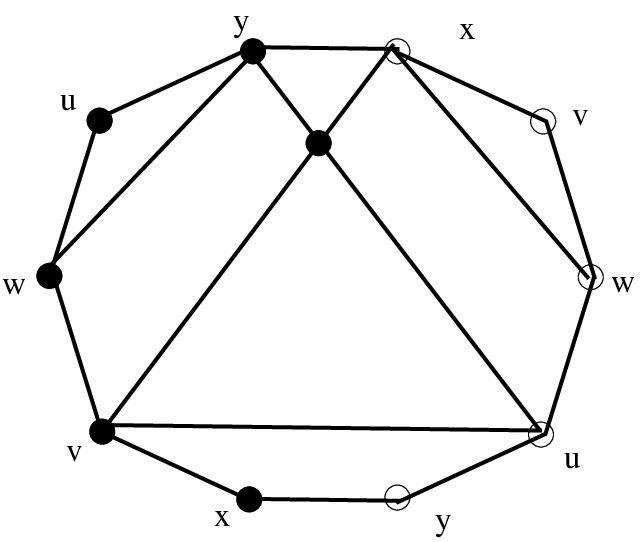}
\caption{Some modified face graph representations for $G_{b,4}$.} 
\label{f:OctagonEtc}
\end{figure}
\end{center}

One can argue in the same way when $r=4$ and when $r=5$ to obtain the forms  of $B_0$ indicated in Figure \ref{f:OctagonEtc}. Their associated surface graphs are isomorphic to $G_{b,4}$, completing the proof.
\end{proof}

The following general construction of a critical embedded 6-cycle $c_g$ with
a maximality property will be used in the next two proofs. We refer to it in the proofs as the \emph{maximal critical walk construction}.

Suppose that $G$ in $\fP_2$ or $\fP_3$ is irreducible and contains a critical $6$-walk $c_e$ with an $FF$ edge $e$. Let $G=G_e\cup A_e$ be the associated decomposition where $c_e$ is the boundary walk $\pi(\partial D_e)$ of an embedded triangulated disc $\pi(D_e)$ in a fixed triangulated surface graph $S$ for $\P$. If $G_e$ is equal to the irreducible $G_{b,1}$ then the construction stops. Otherwise $G_e$ has an $FF$ edge $f$ and we continue.
The edge $f$ must lie on a critical embedded cycle $c_f$ of $G$ which in turn is the boundary walk $\pi(\partial D_f)$ of an embedded triangulated disc $\pi(D_f)$. If the interior of the embedding of $D_e^\circ \cup D_f^\circ)$ is an open disc then Lemma \ref{l:blendingcriticals} applies and one obtains a critical  6-walk with embedded triangulated disc $\pi(D_e\cup D_f)$. Continuing, the construction process stops, either with $G_g= G_{b,1}$ and $c_g$ a repeated essential 3-cycle, or with $c_g$  a critical  6-walk $c_g$ with $FF$ edge $g$ and associated decomposition $G=G_g \cup A_g$  with the following properties:  (i) There is a face in $G_g$, and (ii) if $c_h$ is a critical walk through an $FF$ edge $h$ of $G_g$ then the interior of $\pi(D_g)\cup \pi(D_h)$ is not an open disc. This means that $c_h$ separates holes in the manner shown in Figure \ref{f:2criticalsFailingHypothesis}(ii).

\begin{prop}\label{p:2holesIrred}
Let  $G$ be an irreducible surface graph in $ \fP_2$. Then 
$G=G_{b,2}$.
\end{prop}

\begin{proof}
By Proposition \ref{p:noFacesB} $G$ contains $FF$ edges.
Also, by Corollary \ref{c:propercontainment} $G$ contains no $FF$ edges on critical walks of length 4. Moreover, since $G$ is in  $ \fP_2$ it follows from Lemma \ref{l:mholelemma} that no $FF$ edge lies on a critical walk of length 5. Thus there is an $FF$ edge $e$ on a critical embedded 6-cycle and the maximal critical walk construction applies to give the critical walk $c_g$. 

We consider first the second outcome of the construction. Since $G_g$ is not equal to $G_{b,1}$ it is reducible with an $FF$ edge $h$ with $G_g/h$ in $\fP_1$. By the irreducibility of $G$ and the construction of $G_g$ the edge $h$ lies on a critical embedded 6-cycle $c_h$ for which the open set $\pi(D_g^\circ)\cup \pi(D_h^\circ)$ is not an open disc. The possibilities for this are limited as we now show.

The walk $c_h$, with length 6, necessarily has two subwalks that are interior to $c_g$ (in the usual sense) one with length 2 and the other of length 2 or 3.
Suppose first that both subwalks are of length 2. 
Then $c_h$ could be a repeated essential 3-cycle with two edges interior to $c_g$, as illustrated in Figure \ref{f:2holesPossiblesBandA}(i) and (ii). Moreover, further forms 
are only possible by concatenating certain pairs of essential 3-cycles which share a vertex. So we may  assume that $h$ lies on an essential 3-cycle as shown in Figure \ref{f:2holesPossiblesBandA}(i) or in Figure \ref{f:2holesPossiblesBandA}(ii).
\begin{center}
\begin{figure}[ht]
\centering
\includegraphics[width=4cm]{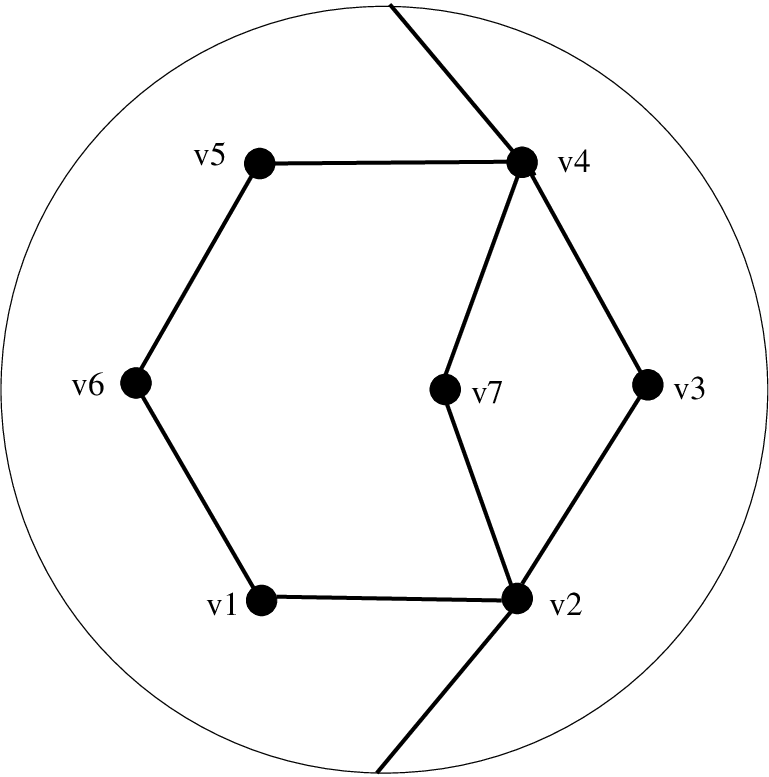}\quad \quad
\includegraphics[width=4cm]{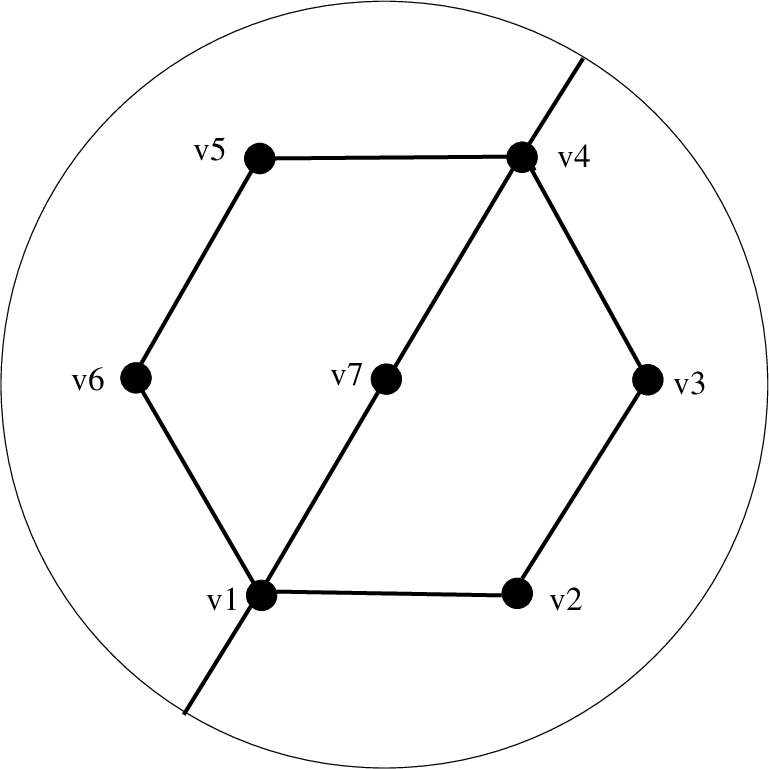}  
\caption{Two forms of an essential 3-cycle with 2 edges interior to the embedded $6$-cycle $c_g$.} 
\label{f:2holesPossiblesBandA}
\end{figure}
\end{center}
In the first case note that if there is an additional $FF$ edge of $G_g$, with faces in $G_g$, that is incident to a vertex of $c_g$ then its critical embedded 6-cycle must contain a ``diagonal" essential 3-cycle, as in the second case. So we may assume that $G_g$ only has the edge $v_2v_4$ as an $FF$ edge. However, this implies that $v_2v_4v_3$ is a facial 3-cycle, a contradiction.

In the second case, illustrated by  Figure \ref{f:2holesPossiblesBandA}(ii), we may assume, by Corollary \ref{c:propercontainment} and symmetry, that the boundary of the 5-cycle hole is $v_4v_5v_6v_1v_7$. Note that the 4-cycle for $c_1$ cannot be incident to both $v_1$ and $v_4$.  Thus, if neither $v_7v_2$ or $v_7v_3$ exists then there would be an edge $v_7w$ with $w$ strictly interior to the 5-cycle $c=v_1v_2v_3v_4v_7$. Also $v_7w$ would not be an $FF$ edge since it cannot lie on a critical walk. So, without loss of generality we can assume that $c_1$ is the walk $v_7wv_3v_4$. It follows that there is an $FF$ edge $v_7w'$ with $w'$ interior to is an $FF$ edge. However, it fails to lie on a critical walk, a contradiction.

In the second case then we may assume that $v_7v_2 $ exists.   
Thus $G$ contains the subgraph shown in the $\P$-diagram
of Figure \ref{f:2holesBasicC}, and $G$ is obtained by adding edges and vertices that are exterior to $c_g$.
Consider an $FF$ edge of $G_g$ (with faces in $G_g$) which is incident to $v_3, v_5$ or $v_6$. This edge must lie on a critical walk separating the holes of $G$ in the manner of Figure \ref{f:2criticalsFailingHypothesis}(ii) and this is not possible in each case. It follows that $v_6=v_3$ and that $v_5=v_2$. Thus the edges $v_4v_2$ and $v_4v_6$ exist and so $G=G_{b,2}$.
\begin{center}
\begin{figure}[ht]
\centering
\includegraphics[width=4cm]{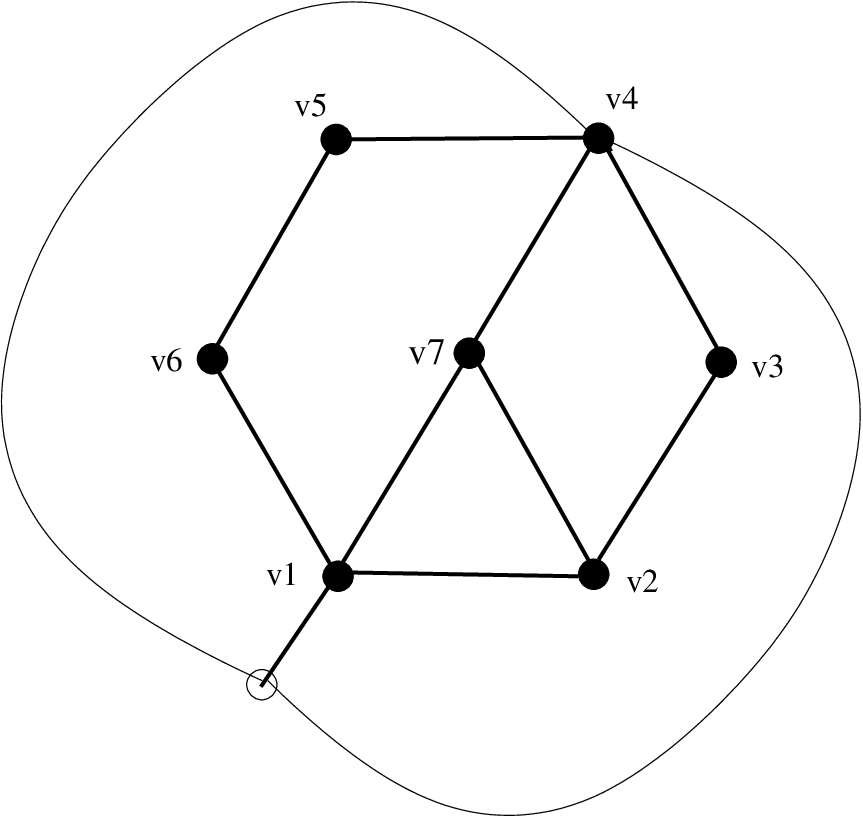}  
\caption{A surface subgraph of $G$ when the essential 3-cycle has 2 edges interior to $c_g$.} 
\label{f:2holesBasicC}
\end{figure}
\end{center}

Suppose next that the subwalks of $c_h$ have lengths 2 and 3 and $h$ does not lie on an essential 3-cycle. Then $c_h$ has the form as shown in Figure \ref{f:2holesBasicC_23case}, for a relabelling of $v_1,\dots v_6$, where $w$ is distinct from $u, v$ and where the walk $v_1, v_2,w,v_4,v_5,v,u$ is the boundary walk of an embedded triangulated disc in $G$.
By  Corollary \ref{c:propercontainment} the cycles  $v_2wv_4v_3$ and $v_1uvv_5v_6$ correspond to the boundary walks of the holes of $G$. It follows that there is an $FF$ edge incident to $u$ or to $v$. Thus does not lie on a critical embedded cycle and so this case does not arise. 
\begin{center}
\begin{figure}[ht]
\centering
\includegraphics[width=4cm]{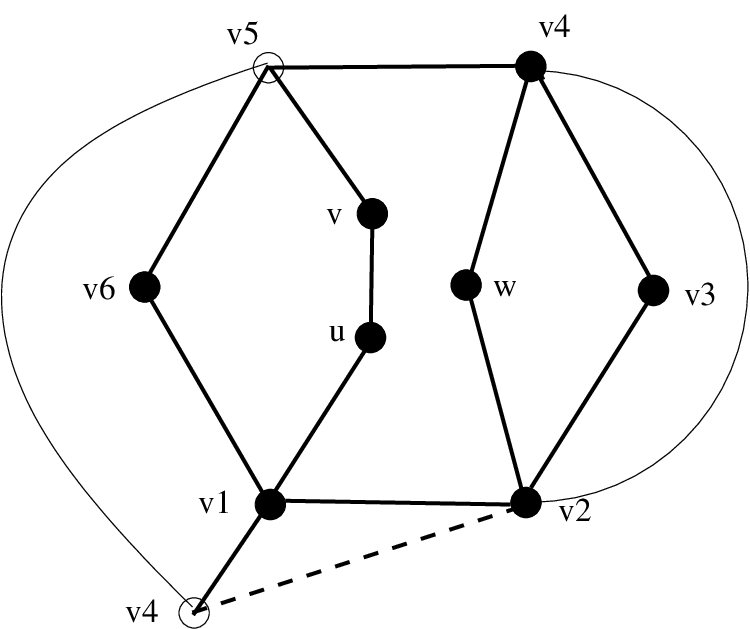}  
\caption{A surface subgraph of $G$ with an embedded critical 6-cycle 
$v_1uvv_5wv_4$.} 
\label{f:2holesBasicC_23case}
\end{figure}
\end{center}
\begin{center}
\begin{figure}[ht]
\centering
\includegraphics[width=4cm]{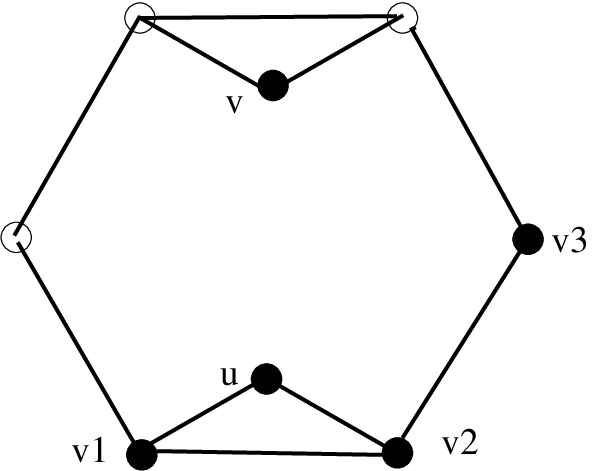}  \quad \quad
\includegraphics[width=4cm]{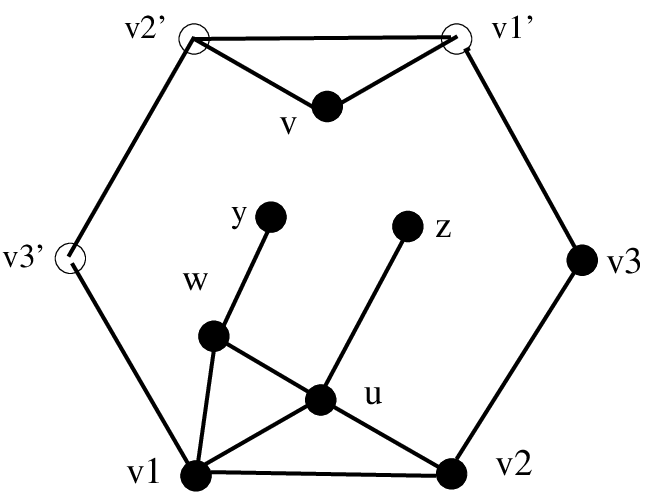}  
\caption{(i) A face graph for $G_g=G_{b,1}$ with added faces $v_1v_2u, v_1v_2v$ where $g=v_1v_2$.  (ii) Implied edges when $v_1$ is not incident to a hole of the face graph $(B_0,\lambda)$ for $G$.} 
\label{f:2holes3cycleCase}
\end{figure}
\end{center} 

Suppose, finally, that the maximal critical walk construction stops with $G_g=G_{b,1}$. Then there is a modified face graph $(B_0,\lambda)$ for $G$ in which the outer boundary of $B_0$ embeds as $c_g$, with both faces of the $FF$ edge $g$ interior to $c_g$. This is illustrated in Figure \ref{f:2holes3cycleCase}(i) where $g=v_1v_2$. If the edge $uv_3$ exists then $uv_2$ is an $FF$ edge and $v_1uv_3v_1v_2v_3$ is a critical embedded 6-cycle, $c$ say. With the pair $uv_2, c$ playing the roles of $g, c_g$ above it follows from our previous arguments that $G=G_{b,2}$. 

We may assume then that none of the edges $uv_3,uv_3',vv_3,vv_3'$ exist.
Since there are 2 holes in the modified face graph $(B_0,\lambda)$ for $G$ it follows that at least one of the vertices $v_1,v_2,v_1',v_2'$ is not incident to either hole of $B_0$. For if not then there is an embedded planar 5-cycle containing both holes of $G$. By symmetry we assume that $v_1$ is such a vertex. 
It follows that there is a face  $v_1uw$ with $w$ interior to $c_g$. Both $v_1w$ and $v_1u$ are $FF$ edges and so lie on a critical embedded 6-cycles, $c_w, c_u$ say, since $G$ is in $\fP_2$. Edges $v_1w, wy$ (resp. $v_1u, uz$) are included in a subwalk of $c_w$ (resp. $c_u$) from $v_1$ to another vertex of $\partial B_0$. See Figure \ref{f:2holes3cycleCase}(ii). We show that in all cases there exists a reducible edge, that is, an $FF$ edge that is not on a critical walk.

If $y=z$ then $uw$ is $FF$ and so lies on a critical walk. The only possibility for this is that $wv_2'$ exists. 
Thus the interior of the 4-cycle $v_1wv_2'v_3'$ is triangulated with an edge incident to $w$, and this edge is reducible. If $y=v_2'$ then, for the same reason, there is a reducible edge. If $y=v$ then $zv_1'$ must exist and there are then three 5-cycle faces one of which is triangulated, and in each case there is a reducible edge. The same conclusion holds if the edge $yv_1'$ exists. Thus the critical 6-cycle $c_g$ cannot exist and the proof is complete.  
\end{proof}

The next proof is similar in style to the previous proof. However this time the outcome $G_g=G_{b,1}$ of the maximal critical walk construction leads to the  identification of an additional irreducible.

\begin{prop}\label{p:444irreducible} 
Let $G$ be an irreducible surface graph in $\fP_3$ with $FF$ edges.
Then $G$ is one of the base graphs $G_{b,5}, \dots , G_{b,8}$.
\end{prop}

\begin{proof} 
We consider $G$ to be determined by a triangulated surface graph for $\P$ and 3 interior-disjoint embedded triangulated discs $\pi(D_1), \pi(D_2), \pi(D_3)$ with boundary walks $c_1, c_2, c_3$ which are critical embedded 4-cycles. By Corollary \ref{c:propercontainment} there are no critical walks of length 4.

Suppose that there is an $FF$ edge on a critical embedded 6-cycle and that
$G=G_g \cup A_g$ is the decomposition given by the maximal critical walk construction. Consider first the outcome $G_g\neq G_{b,1}$. Then $G_g$ contains an $FF$ edge $h$, with both faces in $G_g$. This edge must lie on a critical walk $c_h$ in $G$ and, by the maximal property of $G_g$  there are 2 subwalks of $c_h$ each with at least 2 edges interior to $c_g$. 
\begin{center}
\begin{figure}[ht]
\centering
\includegraphics[width=4cm]{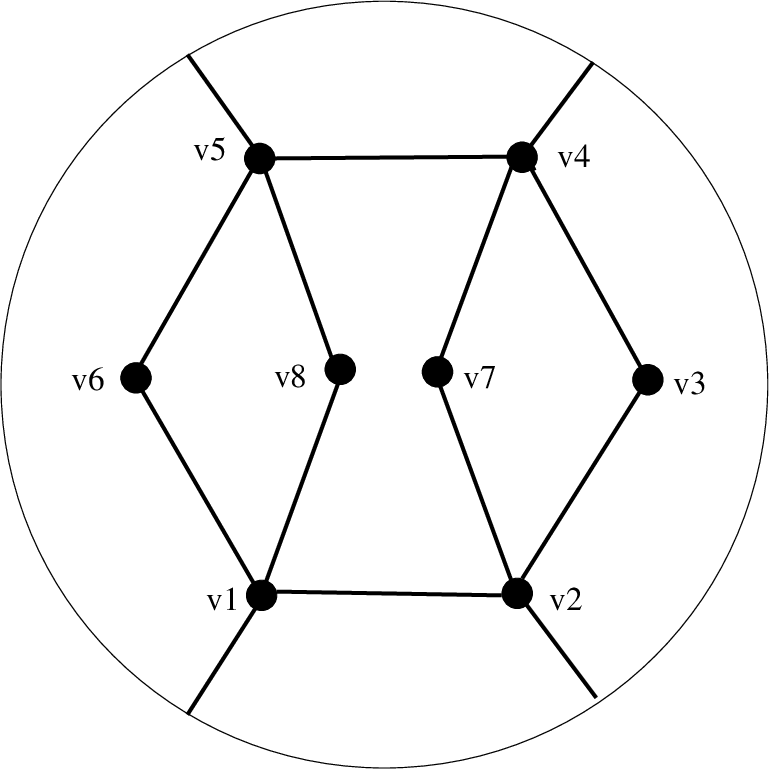}  
\caption{An additional form for $c_h$.} 
\label{f:3holes_6cycle_ch}
\end{figure}
\end{center}

 We consider the case that $|c_h|=6$ and both subwalks are of length 2. Then in addition to the forms implied by Figure \ref{f:2holesPossiblesBandA}, as repeated  or concatenated essential 3-cycles, we also have the form shown in Figure \ref{f:3holes_6cycle_ch} where the subwalks have no coincident vertices and $v_7\neq v_8$. However, by Corollary \ref{c:propercontainment}  the 4-cycles $v_1v_8v_5v_6$ and $v_2v_3v_4v_7$ are boundary cycles determining $c_1$ and $c_2$ say, while the interior of the 6-cycle $v_1v_2v_7v_4v_5v_8$ is partially triangulated in $G$ with a single 4-cycle face associated with $c_3$. In particular the planar embedded 6-cycle $c_h$ cannot be critical and so this case does not arise.

We may assume then that $h=v_1v_4$  lies on an essential 3-cycle $v_1v_7v_4$ as shown in Figure \ref{f:2holesPossiblesBandA}(ii). We may assume also that the cycle $c_1$ is contained in the right hand side 5-cycle, $c=v_1v_2v_3v_4v_7$, and $c_2, c_3$ are contained in the other 5-cycle (in the usual inclusion sense). The argument in the previous proof applies, even though critical walks of length 5 might now be present, as we now show. 
The 4-cycle for $c_1$ cannot be incident to both $v_1$ and $v_4$ and so if neither $v_7v_2$ or $v_7v_3$ exists then there would be an edge $v_7w$ with $w$ strictly interior to  $c=v_1v_2v_3v_4v_7$. Also $v_7w$ cannot be an $FF$ edge since it does not lie on a critical walk of length 5 or 6. So we can assume that $c_1$ is the walk $v_7wv_3v_4$. It follows that there is an $FF$ edge $v_7w'$ with $w'$ interior to is an $FF$ edge. However, it fails to lie on a critical walk of length 5 or 6 and this contradiction shows that we may assume that $v_7v_2$ exists, so that $G$ contains the surface subgraph of Figure \ref{f:3holesBasicC}.
\begin{center}
\begin{figure}[ht]
\centering
\includegraphics[width=4.5cm]{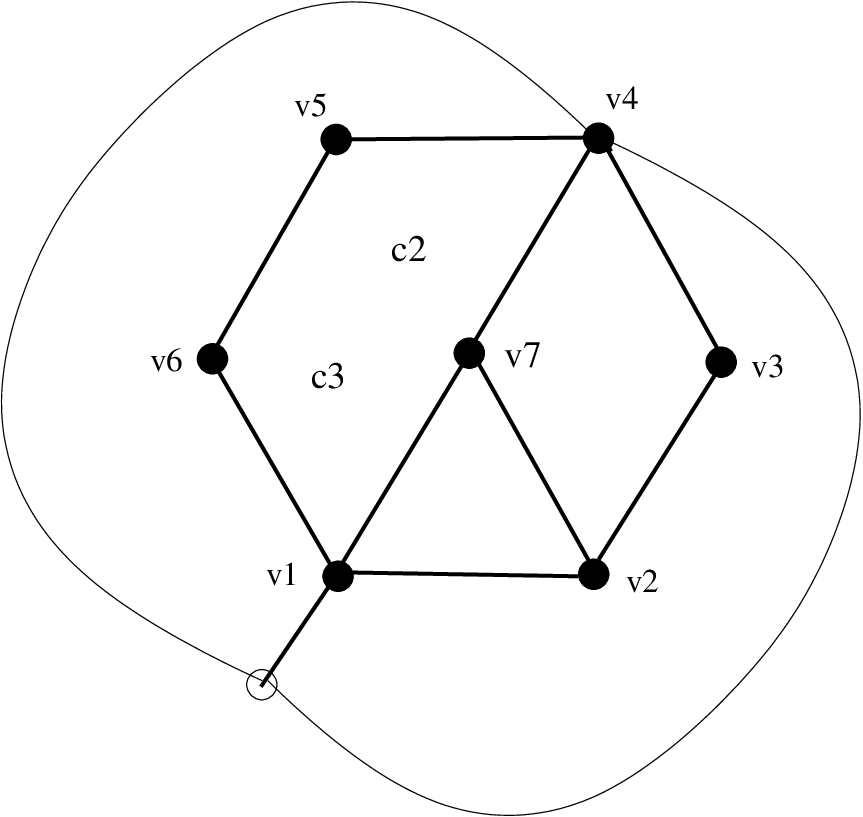} 
\caption{A surface subgraph of $G$ when an essential 3-cycle has 2 edges interior to $c_g$.} 
\label{f:3holesBasicC}
\end{figure}
\end{center}
Suppose, by way of contradiction, that there is an $FF$ edge interior to the embedded 5-cycle $v_1v_7v_4v_5v_6$ containing $c_1$ and $c_2$. Then it lies on a critical embedded cycle which must pass through $v_7$. Since this is not possible  there is a single degree 3 vertex interior to the embedded cycle. It also follows, as in the previous proof, that $v_1v_4$ is the only $FF$ edge of $G_g$. Thus $G$ is equal to one of 5 surface graphs, namely $G_{b,6}, G_{b,7}, G_{b,8}$, or one of the isomorphs of $G_{b,7}$ and $G_{b,8}$ shown in Figure \ref{f:irred_PlusExtra}.
\begin{center}
\begin{figure}[ht]
\centering
\includegraphics[width=4.5cm]{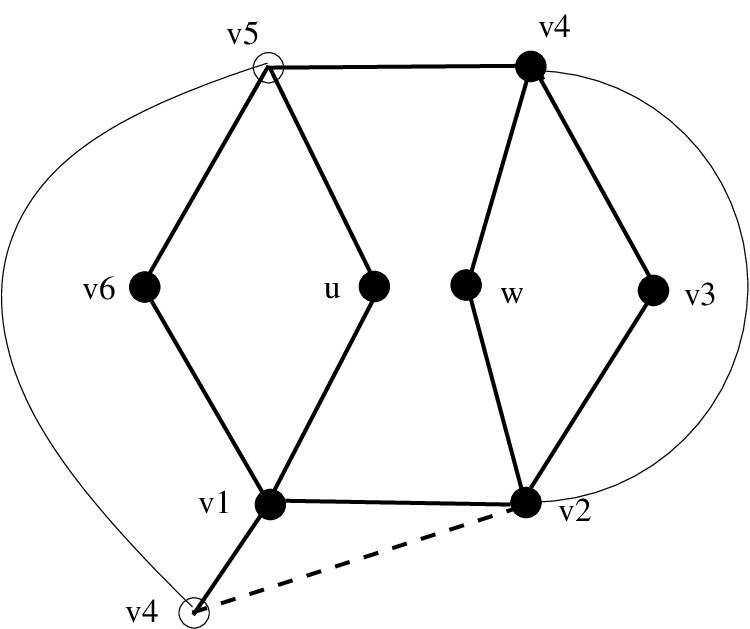}  
\caption{A surface subgraph of $G$ with $c_h$ on a critical embedded 5-cycle.} 
\label{f:3holes_5cycle_chB}
\end{figure}
\end{center}

The case that $c_h$ has length 6 with subwalks of lengths 2 and 3 with $h$ not lying on an essential 3-cycle does not arise. This is because the argument for this case in the proof of Proposition \ref{p:2holesIrred} also holds when $G$ is in $\fP_3$.
Since $G$ is in $\fP_3$ it is possible that $c_h$ has length 5 with subwalks of lengths 2 as shown in Figure \ref{f:3holes_5cycle_chB}.
In view of previous cases we may assume that the edge $h=v_1v_4$,  does not lie on an essential 3-cycle. Thus the edge $v_1w, uv_4$ do not exist. Also $uv_2$ and $v_5w$ do not exist, by the simplicity of $G$. It follows that there is an $FF$ edge $uz$ with $z$ interior to $v_1v_2wv_4v_5u$. This edge cannot lie on a critical walk of length 5 or 6 contradicting irreducibility in this case.

Next we consider the outcome $G_g=G_{b,1}$ in the construction of $G_g$ and show that $G=G_{b,5}$ in this case. Figure \ref{f:2holes3cycleCase}(i) illustrates a modified face graph for $G_g$ augmented by the two faces of $g=v_1v_2$. The other vertices of $G$ are interior to $c_g$ and the walks $c_1, c_2, c_3$ in $G$ derive from 3 quadrilateral faces of the associated modified face graph $(B_0,\lambda)$  for $G$.

Suppose first that $v_1$ does not belong to any of the walks $c_1, c_2, c_3$. Then  the edges that are interior to $c_g$ and incident to $v_1$, say $v_1u, v_1u_1 ,\dots , v_1u_r$, are $FF$ edges. See Figure \ref{f:3holesInteriorVertexCase}(i). 
Each 4-cycle face of $(B_0,\lambda)$ includes at most one of the edges $v_2u, uu_1, u_1u_2, \dots u_rv_3'$ and so if $r>1$ at least one of the edges $uu_1, \dots ,u_{r-1}u_r$
is an $FF$ edge, $e$ say. Suppose $e=uu_1$. To be on a critical walk requires $u_1v_2'$ to exist (since $uv_3$ does not exist). It follows that $u_jv_2'$ must exist for each $j$, since each edge $v_1u_i$ lies on a critical walk, and so $u_1u_2$ is a reducible edge. The argument is the same for each $u_iu_{i+1}$. Thus $r=1$ and the degree of $v_1$ in the face graph $(B_0,\lambda)$ is 4. Since $uu_1$ is incident to a 4-cycle face it follows readily that irreducibility implies that $G=G_{b,5}$. 
\begin{center}
\begin{figure}[ht]
\centering
\includegraphics[width=4cm]{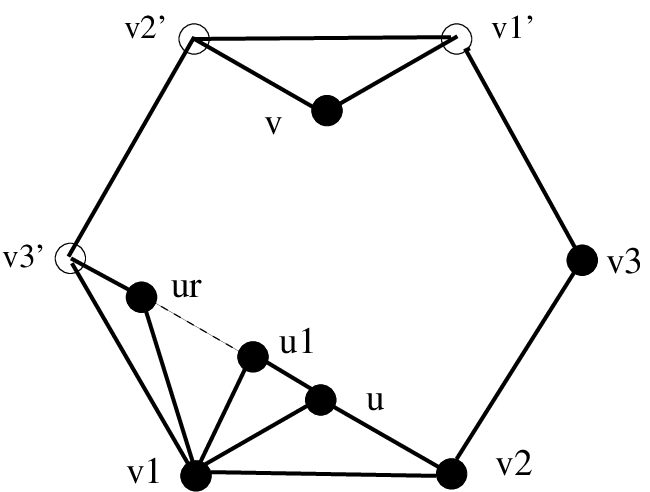} \quad \quad 
\includegraphics[width=4cm]{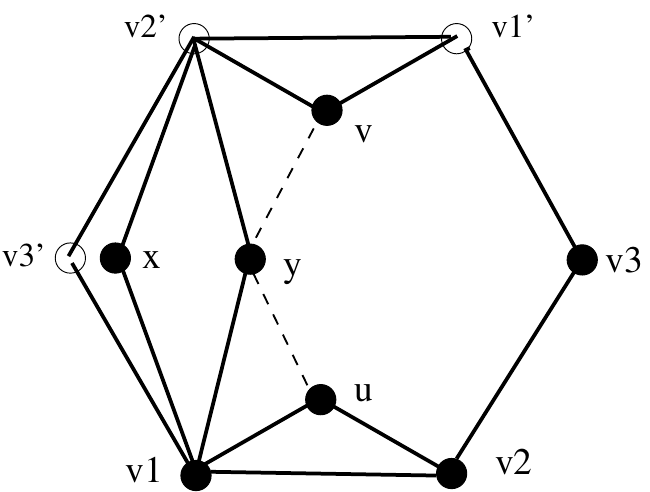}   
\caption{(i) When $v_1$ is not incident to a 4-cycle hole. (ii) When $v_1, v_2, v_1', v_2'$ are each incident to a 4-cycle hole.} 
\label{f:3holesInteriorVertexCase}
\end{figure}
\end{center} 

It remains to consider the case that each vertex $v_1, v_2, v_1', v_2'$ belongs to at least one of the walks $c_1, c_2, c_3$. Relabelling we may assume that $c_1$ is the embedding of a 4-cycle $v_1yv_2'x$, as in Figure \ref{f:3holesInteriorVertexCase}(ii), where, by simplicity, $y\neq u, v$. It follows that $x= v_3'$ 
Since $v_3$ does not have degree 2 neither remaining 4-cycle hole can be incident to both $v_2$ and $v_1'$. Also it is not possible for incidence to the pair $v_1, v_2$ or to the pair $v_1', v_2'$. It follows that all further possible edges of the form $v_1z$ or of the form $v_2'x$, are $FF$ edges. In particular we note that $v_1u$ is an $FF$ edge. On the other hand, given that $c_2$ is incident to $v_1'$ and $c_3$ is incident to $v_2$, it also follows that  $v_u$ cannot lie on a critical walk of length 5 or 6, a contradiction. 

This contradiction completes the proof when critical walks of length 6 exist. It remains to show that there is no irreducible surface graph $G$ in $\fP_3$ with $FF$ edges which only lie on critical walks of length 5. In this case consider one such edge $g$ with critical walk $c_g$ and associated subgraph $G_g$. Since $G_g$ is in $\fP_2$ it has an $FF$ edge $h$ and so this lies on a critical 5-walk $c_h$. Let $\pi(D_g), \pi(D_h)$ be the embedded triangulated discs in $\P$ with boundary walks $c_g, c_h$. Since $G$ is finite we may assume that $D_h$ does not contain $D_g$, for otherwise we could replace $g$ by $h$, continuing with similar replacements if necessary. If the the union of
the open sets $\pi(D_g^\circ), \pi(D_h^\circ)$
is an open disc then by Lemma \ref{l:blendingcriticals} the associated boundary walk is a critical walk. In fact it is a critical 6-walk, since it contains $c_1,c_2,c_3$, and this is contrary to our assumptions.
On the other hand if the union is not an open disc then $c_h$ has 2 subwalks interior to $c_g$ which is not possible for a 5-walk.
\end{proof}

\section{Constructibility and 3-rigidity}
\label{s:mainproof}

Combining results of the previous sections we obtain the construction theorem, Theorem \ref{t:construction}  and the proof of Theorem \ref{t:projectiveA} which we repeat here as Theorem \ref{t:projectiveArepeat}.

\begin{thm}\label{t:construction}
Let $G\in \fP_k$, for $k=1,2$ or 3. Then $G$ is constructible by a finite sequence of planar vertex splitting moves from one of the eight irreducible surface graphs $G_{b,1}, \dots , G_{b,8}$.
\end{thm}

\begin{thm}\label{t:projectiveArepeat}
Let $G$ be a simple graph 
associated with a  partial triangulation
of the real projective plane. 
Then $G$ is minimally $3$-rigid if and only if $G$ is $(3,6)$-tight.
\end{thm}

\begin{proof}
Let $H$ be the graph determined a partial triangulation of the real projective plane. If $H$ is minimally 3-rigid then it is well-known that $H$ is necessarily $(3,6)$-tight \cite{gra-ser-ser}. 

Suppose on the other hand that $H$ is a $(3,6)$-tight graph which is embeddable in $\P$. If this embedding is topologically contractible then $H$ is a planar graph which is $(3,6)$-tight. Such graphs are known to be 3-rigid, since $H$ is either a triangle or is the graph of a triangulation of the sphere.  On the other hand if the embedding is not topologically contractible then the associated surface graph, $G(H)$ say, belongs to $\fP_1, \fP_2$ or $\fP_3$. By Theorem \ref{t:construction} the graph $H$ is constructible by planar vertex splitting moves from one of eight irreducible  graphs, each of which has fewer than $8$ vertices. It is well-known that all $(3,6)$-tight graphs with fewer than $8$ vertices are minimally 3-rigid.
Since vertex splitting preserves minimal 3-rigidity (Whiteley \cite{whi}) it follows that $G$ is minimally 3-rigid. 
\end{proof}

\bigskip

{\bf Acknowledgements.} {This research was supported  by the EPSRC grant EP/P01108X/1, for the project \emph{Infinite bond-node frameworks}, and by a visit to the Erwin Schroedinger Institute in September 2018 in connection with the workshop on \emph{Rigidity and Flexibility of Geometric Structures}.
We thank referees for comments which have improved the presentation.}

\bibliographystyle{abbrv}
\def\lfhook#1{\setbox0=\hbox{#1}{\ooalign{\hidewidth
  \lower1.5ex\hbox{'}\hidewidth\crcr\unhbox0}}}

\vspace{.5in}

\section{Appendix: The uncontractible surface graphs.}
We give an alternative systematic approach to the identification of the uncontractible embeddings of $(3,6)$-tight graphs. As we have remarked these coincide with the topologically contractible embedding of $K_3$ and the 8 topologically uncontractible embeddings provided by Theorem \ref{t:construction}. This is because these irreducible embeddings happen to be uncontractible. We present them here anew in Figures  \ref{f:irreduciblesA}, \ref{f:irreduciblesB}, \ref{f:irreduciblesC},
with new notation which reflects the invariant of Definition \ref{d:holedegree} which we use as an organising principle for their determination.
Specifically, we write $G^h_n, G^h_{n,\alpha}, G^h_{n,\beta} $ where $n$ is the number of vertices and $h=h(G)$ indicates the \emph{minimum hole incidence degree}.

\begin{definition}\label{d:holedegree} Let $v$ be a vertex of the surface graph $G$ in $\fP_k$ for some $k=1,2,3$. Then (i) $\deg_F(v)$ is the number of facial 3-cycles incident to $v$, (ii)  
$
\deg_h(v) = \deg(v)-\deg_F(v)
$
is the \emph{hole incidence degree} for $v$, and (iii) 
$h(G)= \min_v \operatorname{deg}_h(v)$ is the \emph{minimum hole incidence degree}.
\end{definition}



\begin{center}
\begin{figure}[ht]
\centering
\includegraphics[width=2.5cm]{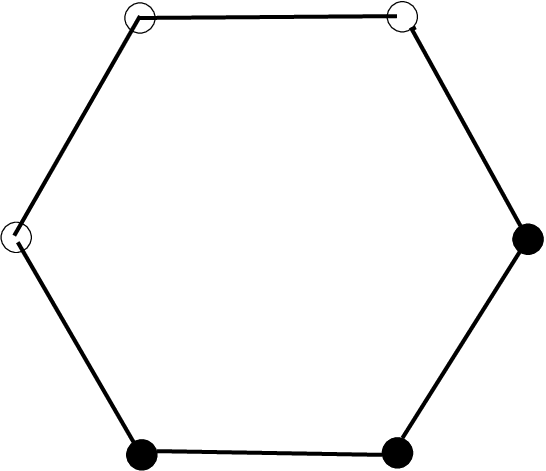}
\quad \quad
\includegraphics[width=2.5cm]{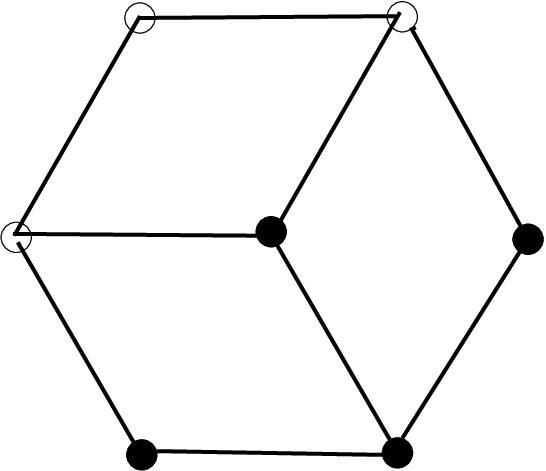}
\caption{$G^2_3=G_{b,1}$ and $G^3_4=G_{b,3}$.}
\label{f:irreduciblesA}
\end{figure}
\end{center}
\begin{center}
\begin{figure}[ht]
\centering
\includegraphics[width=3cm]{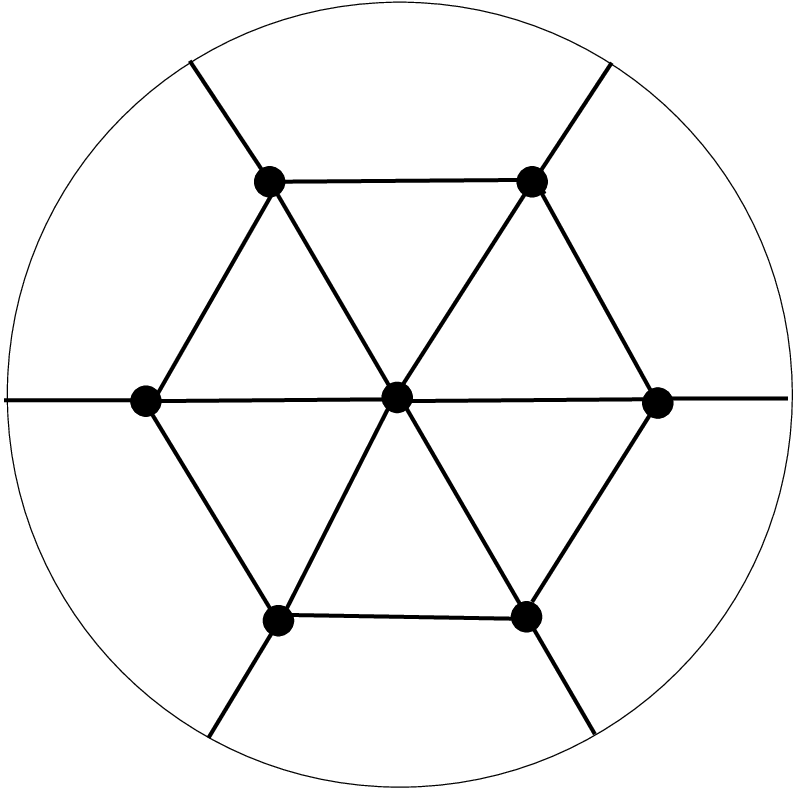}\quad \quad \quad
\includegraphics[width=3cm]{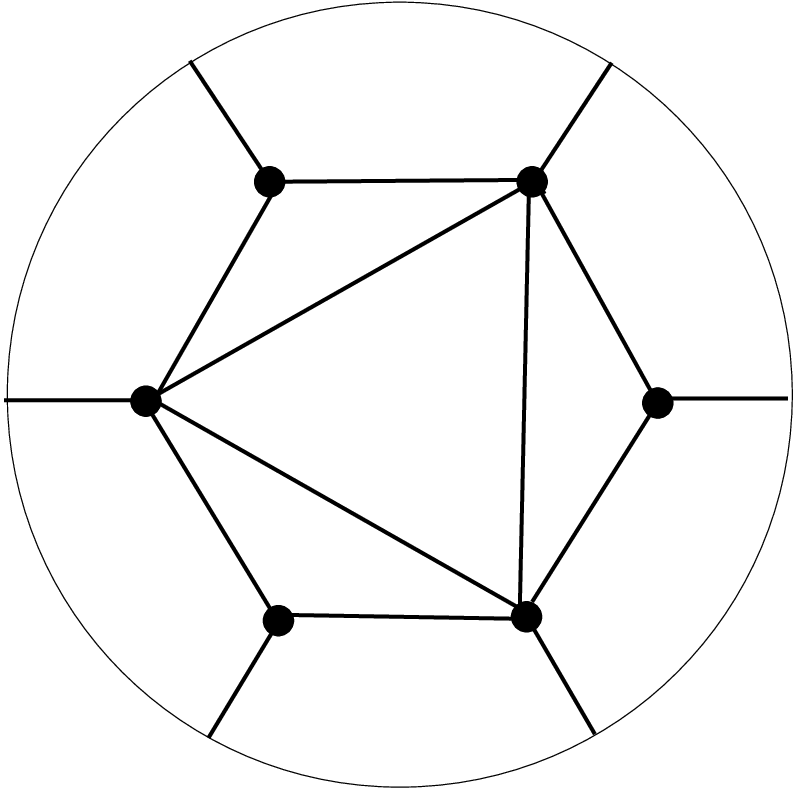}\quad \quad \quad
\includegraphics[width=3cm]{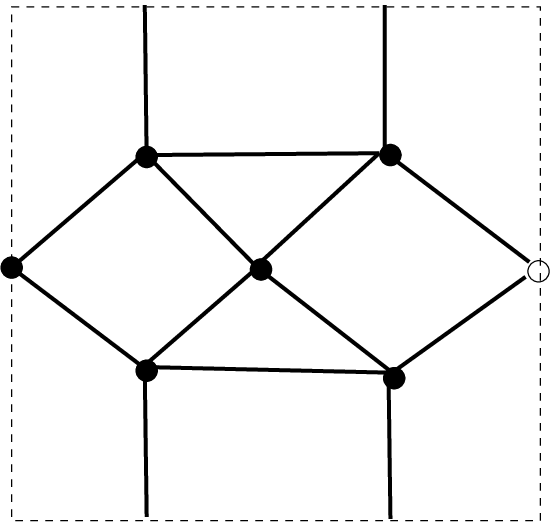}
\caption{$G^0_7= G_{b,5}$, with an interior vertex, $G^2_{6,\alpha}=G_{b,6}$ with a ``ring" of 4-cycle faces,  and  $G^2_{6,\beta}=G_{b,4}$ with no $FF$ edges.} 
\label{f:irreduciblesB}
\end{figure}
\end{center}
\begin{center}
\begin{figure}[ht]
\centering
\includegraphics[width=3cm]{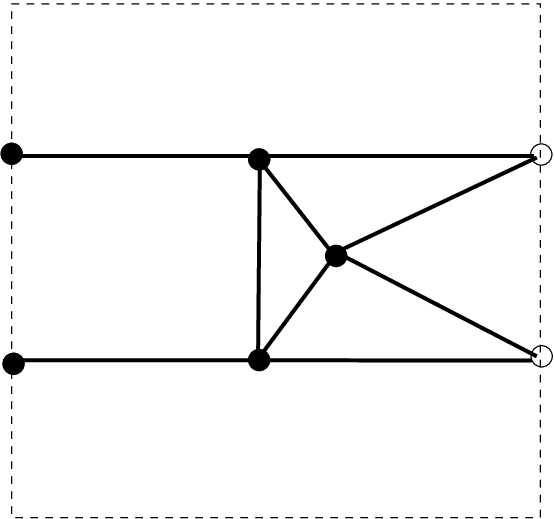}\quad \quad \quad
\includegraphics[width=3cm]{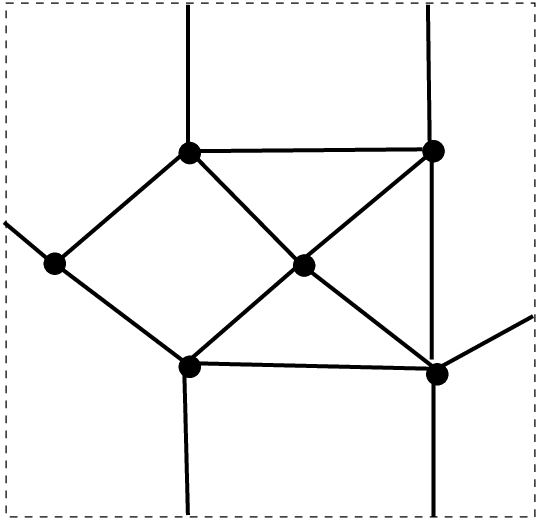}\quad \quad \quad
\includegraphics[width=3cm]{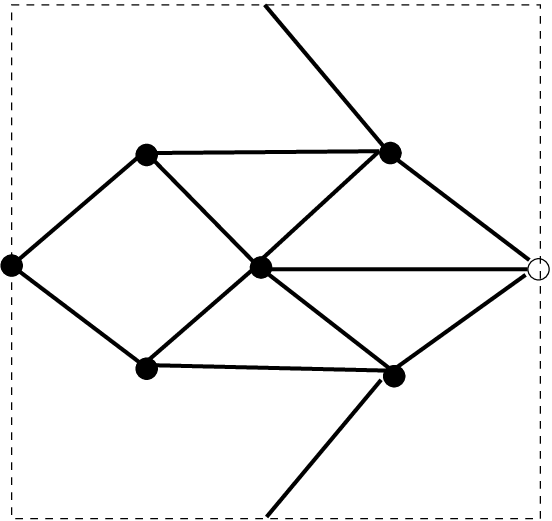} 
\caption{$G^1_5=G_{b,2}$ with 2 holes, $G^1_{6,\alpha}=G_{b,7}$ with 2 $FF$ edges and  $G^1_{6,\beta}=G_{b,8}$ with 3 $FF$ edges.} 
\label{f:irreduciblesC}
\end{figure}
\end{center}




We first note the following proposition identifying the unique uncontractible surface graph in $\fP_k$ with $h(G)=0$ which is to say that there is an interior vertex not contained in the boundary of $G$.

\begin{prop}\label{p:interiorvertex}
Let $G$ be an uncontractible surface graph in $\fP_k$, for $k=1, 2$ or $3$ with $h(G)=0$. Then $k=3$ and $G$ is the hexagon graph $G^0_7$.
\end{prop}

\begin{proof}
Let $v_1$ be the interior vertex of $G$. Since it is incident to an $FF$ edge and this edge necessarily lies on an essential 3-cycle, in view of Lemma \ref{l:planar3cycle}, it follows that $G$ has a modified face graph representation $(B_0,\lambda)$, with 6-cycle boundary, for which the faces incident to $v_1$ provide edges forming paths $\pi_1, \pi_2$ from $v_2$ to $v_3'$ and from $v_2'$ to $v_3$ respectively, as indicated  in Figure \ref{f:interiorNEW}. 
\begin{center}
\begin{figure}[ht]
\centering
\includegraphics[width=4cm]{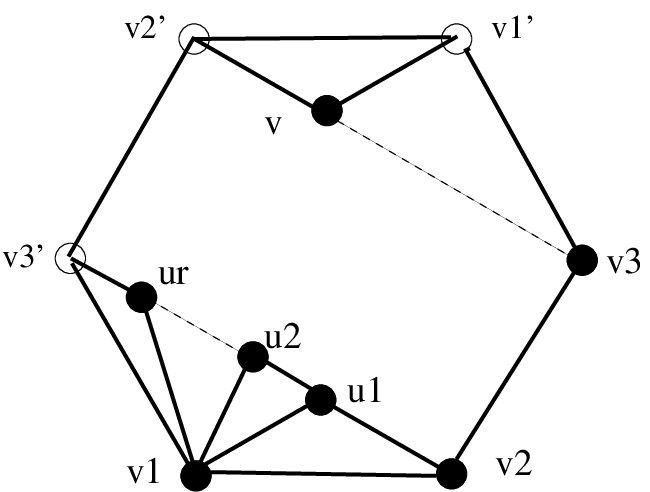}
\caption{Faces incident to $v_1, v_1'$ associated with an interior vertex of $G$.} 
\label{f:interiorNEW}
\end{figure}
\end{center}

If $u_rv_2'$ exists then $v_3'u_r$ is $FF$ and there exist edges $u_rz, zv_3$ for an essential 3-cycle for $v_3'u_r$. Thus the edges $u_iv_3$ necessarily exists for an essential 3-cycle for $v_1u_i$, for $i=1,\dots ,u_{r-1}$. It follows that the $FF$ edge $v_2u_1$ does not lie on an essential 3-cycle if $r\geq 2$. However, for both $r=1,2$, including the possibility that $z=u_1$, there is no completion such that $G$ is (3,6)-tight.

Since $u_rv_2'$ does not exist there exists an $FF$ edge $v_3'y$ of $B_0$ 
together with edges $yz, zv_3$ for an essential 3-cycle. Once again, in the manner of the previous paragraph, there is no completion of $B_0$ to form a modified face graph for a (3,6)-tight surface graph for $\P$.
\end{proof}

The next lemma is  key to the determination of the uncontractible surface graphs $G$ in $\fP_k$ for $k=2$ or $3$. In the proof we use the fact, from Corollary \ref{c:propercontainment}, that it is not possible for $G$ to have a 4-cycle hole whose boundary walk is contained in an embedded 4-cycle of planar type. 

\begin{lemma}\label{degh1lemma} 
Let $G\in \fP_k$, for $k=2$ or 3, be uncontractible with no interior vertex and 
let $v_1$ be a vertex with $\operatorname{deg}_h(v_1)=1$  which lies on the boundary of a 4-cycle hole of $G$ with edges $v_1v_2,v_2v_3,v_3v_4, v_4v_1$. Then $\operatorname{deg}(v_1)= 4$ if $v_1$ is not adjacent to $v_3$ and $\operatorname{deg}(v_1)= 5$ otherwise.
\end{lemma}

\begin{proof}
Let $v_2=w_1,w_2,\dots,w_n=v_4$ be the neighbours of $v_1$ in cyclic order. Since $\operatorname{deg}_h(v_1)=1$, we also have the edges $w_1w_2,\dots,w_{n-1}w_n$. Note that $\deg(v_1)\geq 4$ since if the degree is 3 then the edge $v_1w_2$ is contractible.


{\bf Case (a).}	$v_3\neq w_i$, for every $i\in\{2,\dots,n-1\}$. \\
Suppose, by way of contradiction, that $n\geq 5$. 
For $2\leq i \leq n-1$ it follows from the uncontractibility of the $FF$ edge $v_1w_i$ that there exists an edge $w_iw_r$,
with $1\leq r \leq n$. 
Moreover the 3-cycle $w_iw_rv_1w_i$ cannot be an embedded planar 3-cycle since it is not a face and since it contains no holes of $G$ there would be a contractible edge.

Since there are at most 3 holes and the Case (a) assumption is in force, there is an index $1\leq i<n$ and an associated cycle of edges through $w_i, w_{i+1}, w_s, w_r$, with $r \leq s$, which is triangulated by faces of $G$. 
Thus $w_iw_{i+1}$ is an $FF$ edge.
Since $G$ is uncontractible $w_iw_{i+1}$ lies in a non-facial 3-cycle. Since $v_1w_j$ is also an $FF$ edge for every $j\in\{2,\dots,n-1\}$, it follows that there are just two candidate non-facial 3-cycles: $w_{i-1}w_iw_{i+1}w_{i-1}$ or $w_iw_{i+1}w_{i+2}w_i$.  

\begin{description}
\item[(i)] If $w_iw_{i+1}$ lies on the cycle $w_{i-1}w_iw_{i+1}w_{i-1}$, then the 4-cycle $w_{i-1}v_1w_rw_{i+1}w_{i-1}$ contains strictly the hole boundary $v_1v_2v_3v_4v_1$,  contradicting  Corollary \ref{c:propercontainment}. 
Note that this 4-cycle does contain the hole in our sense, as shown by the shading in Figure \ref{f:case_a} indicating a triangulated disc in $\P$ with boundary equal to this 4-cycle.

\begin{figure}[h!]
\centering
\begin{tikzpicture}
\path [fill=red1] (2.75,1.701) -- (3.5,3) -- (2.75,4.299)--(2,3)-- (2.75,1.701);
\path [fill=red3] (1.4,5)-- (3.1,4.299) --(3.5,3) -- (3.5,1.701) -- (3.3,1) -- (0.5,1) -- (0.5,5) -- (1.4,5);
\path [fill=red3] (4.5,5) -- (4.75,4.299) -- (5.2,1) -- (6.5,1) -- (6.5,5) -- (4.5,5) ;

\draw [red, thick](3.5,4.299)--(3.85,5);
\draw [red, thick](3.75,1)--(3.5,1.701);
\draw[green, thick](4.75,4.299)-- (4.5,5);
\draw[green, thick](3.3,1)--(3.5,1.701);
\draw[blue, thick](3.1,4.299)--(1.4,5);
\draw[blue, thick](5.2,1)--(4.75,4.299);

\draw [->, thick] (0.5,1) -- (4,1);
\draw [thick](4,1) -- (6.5,1);
\draw [<-, thick] (3, 5) -- (6.5,5);
\draw [thick](0.5,5) -- (3,5);
\draw [->, thick] (6.5,1) -- (6.5,3);
\draw [thick](6.5,3) -- (6.5,5);
\draw [<-, thick] (0.5, 2) -- (0.5,5);
\draw [thick](0.5,1) -- (0.5,2);

\draw [thick](2.75,1.701)--(4.75,1.701)--(4.75,4.299)--(2.75,4.299)--(2,3)--(2.75,1.701);
\draw [thick](2.75,1.701) -- (3.5,3) -- (2.75,4.299);
\draw [thick](3.5,3) -- (4.75,1.701);
\draw [thick](3.5,3) -- (4.75,4.299);
\draw [thick](3.5,3) -- (3.5,4.299);
\draw [thick](3.5,3) -- (3.5,1.701);
\draw [thick](3.5,3) -- (3.1,4.299);
\draw [fill] (3.5,3) circle [radius=0.05];
\node [right] at (3.7,2.8) {$v_1$};
\draw [fill] (2.75,1.701) circle [radius=0.05];
\node [below] at (2.75,1.7) {$w_{n}$};
\draw [fill] (3.5,1.701) circle [radius=0.05];
\node [above right] at (3.5,1.7) {$w_{r}$};
\draw [fill] (4.75,4.299) circle [radius=0.05];
\node [right] at (4.75,4.3) {$w_{i+1}$};
\draw [fill] (3.5,4.299) circle [radius=0.05];
\node [above right] at (3.5,4.299) {$w_{i}$};
\draw [fill] (3.1,4.299) circle [radius=0.05];
\node [above] at (3,4.299) {$w_{i-1}$};
\draw [fill] (2.75,4.299) circle [radius=0.05];
\node [left] at (2.7,4.299) {$w_{1}$};
\draw [fill] (2,3) circle [radius=0.05];
\node [left] at (2,3) {$v_{3}$};
\end{tikzpicture}
\caption{The 4-cycle $w_{i-1}v_1w_r w_{i+1}w_{i-1}$ contains strictly the 4-hole $v_1v_2v_3v_4v_1$.}\label{f:case_a}
\end{figure}

\item[(ii)] If $w_iw_{i+1}$ lies on the cycle  $w_iw_{i+1}w_{i+2}w_i$, then, {noting that $w_{i+2}w_i$ is an edge}, we claim  that the 5-cycle $w_iv_1w_rw_{i+1}w_{i+2}w_i$  contains all the holes, which is a contradiction.  
To see this note that by 
Corollary \ref{c:propercontainment} the 4-cycle $v_1w_rw_iw_{i+2}v_1$ contains no holes. See Figure \ref{f:case_a'}. 

\begin{figure}[h!]
\centering
\begin{tikzpicture}
\path [fill=red1] (2.75,1.701) -- (3.5,3) -- (2.75,4.299)--(2,3)-- (2.75,1.701);
\path [fill=red3] (1.4,5)-- (3.1,4.299) --(3.5,3) -- (3.5,1.701) -- (3.3,1) -- (0.5,1) -- (0.5,5) -- (1.4,5);
\path [fill=red3] (4.5,5) -- (4.75,4.299) -- (4.75,3.5) -- (5.2,1) -- (6.5,1) -- (6.5,5) -- (4.5,5) ;

\draw [red, thick](3.1,4.299)--(3.35,5);
\draw [red, thick](3.75,1)--(3.5,1.701);
\draw[green, thick](4.75,4.299)-- (4.5,5);
\draw[green, thick](3.3,1)--(3.5,1.701);
\draw[blue, thick](3.1,4.299)--(1.4,5);
\draw[blue, thick](5.2,1)--(4.75,3.5);

\draw [->, thick] (0.5,1) -- (4,1);
\draw [thick](4,1) -- (6.5,1);
\draw [<-, thick] (3, 5) -- (6.5,5);
\draw [thick](0.5,5) -- (3,5);
\draw [->, thick] (6.5,1) -- (6.5,3);
\draw [thick](6.5,3) -- (6.5,5);
\draw [<-, thick] (0.5, 2) -- (0.5,5);
\draw [thick](0.5,1) -- (0.5,2);

\draw [thick](2.75,1.701)--(4.75,1.701)--(4.75,4.299)--(2.75,4.299)--(2,3)--(2.75,1.701);
\draw [thick](2.75,1.701) -- (3.5,3) -- (2.75,4.299);
\draw [thick](3.5,3) -- (4.75,4.299);
\draw [thick](3.5,3) -- (4.75,3.5);
\draw [thick](3.5,3) -- (3.5,1.701);
\draw [thick](3.5,3) -- (3.1,4.299);
\draw [fill] (3.5,3) circle [radius=0.05];
\node [right] at (3.7,2.8) {$v_1$};
\draw [fill] (2.75,1.701) circle [radius=0.05];
\node [below] at (2.75,1.7) {$w_{n}$};
\draw [fill] (3.5,1.701) circle [radius=0.05];
\node [above right] at (3.5,1.7) {$w_{r}$};
\draw [fill] (4.75,4.299) circle [radius=0.05];
\node [right] at (4.75,4.3) {$w_{i+1}$};
\draw [fill] (4.75,3.5) circle [radius=0.05];
\node [above right] at (4.75,3.5) {$w_{i+2}$};
\draw [fill] (3.1,4.299) circle [radius=0.05];
\node [above] at (3,4.299) {$w_{i}$};
\draw [fill] (2.75,4.299) circle [radius=0.05];
\node [left] at (2.7,4.299) {$w_{1}$};
\draw [fill] (2,3) circle [radius=0.05];
\node [left] at (2,3) {$v_{3}$};
\end{tikzpicture}
\caption{The 5-cycle $w_iv_1w_rw_{i+1}w_{i+2}w_i$  contains all the holes.}\label{f:case_a'}
\end{figure}
\end{description}
These contradictions, together with Lemma \ref{l:degree3lemma} show that $n=4$ in this case.
\medskip

 
{\bf Case (b).}   $v_3= w_{i_0}$, for some $i_0\in\{2,\dots,n-1\}$.\\
Since $G$ is a simple graph and $v_1v_3$ is now an uncontractible $FF$ edge we have $\operatorname{deg}(v_1)\geq 5$. Suppose, by way of contradiction, that $n\geq 6$. As in case (a) there then exists an $FF$  edge $w_iw_{i+1}$ and some vertex $w_r$ providing a facial 3-cycle  $w_iw_{i+1}w_r$. (See Figure \ref{MDsep2}.) The only possible non-facial 3-cycle for the $FF$ edge $w_iw_{i+1}$ is $v_3w_iw_{i+1}v_3$. However, this gives a contradiction since the 4-cycle $w_iv_3v_4v_1w_i$ strictly 
contains the hole $v_1v_2v_3v_4v_1$. Thus $n=5$. 

 \textcolor[rgb]{0,0,1}. 
\begin{figure}[h!]
\centering
\begin{tikzpicture}
\path [fill=red1] (0.5,3)--(1.25,1.701)--(2,3)--(1.25,4.299)--(0.5,3);
\path [fill=red3] (3.5,4.299) -- (2,3)--(1.25,1.701)-- (0.5,3)--(0.5,3) to[out=80,in=180] (1.25,4.6)-- (3,4.6) to[out=0,in=80] (3.5,4.299);

\draw [->, thick] (0.5,1) -- (4.5,1);
\draw [thick](4.5,1) -- (6.5,1);
\draw [<-, thick] (2.5, 5) -- (6.5,5);
\draw [thick](0.5,5) -- (2.5,5);
\draw [->, thick] (6.5,1) -- (6.5,2.5);
\draw [thick](6.5,2.5) -- (6.5,5);
\draw [<-, thick] (0.5, 2.5) -- (0.5,5);
\draw [thick](0.5,1) -- (0.5,2.5);

\draw [thick](0.5,3)--(1.25,1.701)--(2,3)--(1.25,4.299)--(0.5,3);
\draw [thick](1.25,1.701)--(4.5,1.701)--(6.5,3);
\draw [thick](1.25,4.299)--(3.5,4.299)--(4.5,4.299)--(6.5,3);
\draw [thick](4.5,1.701) -- (2,3) -- (4.5,4.299);
\draw [thick](2,3) -- (6.5,3);
\draw [thick](2,3) -- (3.5,4.299);

\draw [blue, thick](0.5,3) to[out=80,in=180] (1.25,4.6);
\draw [blue, thick](1.25,4.6) -- (3,4.6);
\draw [blue, thick](3,4.6) to[out=0,in=80] (3.5,4.299);
\draw [blue, thick] (4.5,4.299) to[out=0,in=90] (6.5,3);
\draw [red, thick](3.5,4.299)--(3.45,5);
\draw [red, thick](4.1,1)--(4.5,1.701);
\draw[green, thick](4.5,4.299)-- (4,5);
\draw[green, thick](3.5,1)--(4.5,1.701);

\draw [fill] (2,3) circle [radius=0.05];
\node [below] at (1.2,2.8) {$v_1$};
\draw [fill] (1.25,1.701) circle [radius=0.05];
\node [below] at (1.25,1.7) {$v_{4}$};
\draw [fill] (1.25,4.299) circle [radius=0.05];
\node [above] at (1.25,4.29) {$v_{2}$};
\draw [fill] (0.5,3) circle [radius=0.05];
\node [right] at (0.5,3) {$v_{3}$};
\draw [fill] (3.5,4.299) circle [radius=0.05];
\node [above right] at (3.5,4.299) {$w_{i}$};
\draw [fill] (4.5,1.701) circle [radius=0.05];
\node [below right] at (4.5,1.701) {$w_{r}$};
\draw [fill] (4.5,4.299) circle [radius=0.05];
\node [above right] at (4.5,4.299) {$w_{i+1}$};
\draw [fill] (6.5,3) circle [radius=0.05];
\node [below right] at (6.5,3) {$v_3$};

\end{tikzpicture}
\caption{The 4-cycle $v_4v_3v_1w_iv_4$ contains strictly the 4-hole $v_1v_2v_3v_4v_1$.}\label{MDsep2}
\end{figure}


\end{proof}

\begin{prop}\label{p:3uncontractibles}
Let $G\in \fP_k$, for $k=1, 2$ or 3, be uncontractible with no interior vertex. If there exists a vertex $v_1\in V(G)$ with $\operatorname{deg}_h(v_1)=1$ then $G$ is 
one of the surfaces graphs $ G^1_{6,\alpha}, G^1_{6,\beta}, G^1_5$. 
\end{prop}
\begin{proof}

{\bf Case (a).} 
Assume first that $v_1$ lies on the 4-cycle boundary of the hole $H_1$, with  vertices $v_1, v_2, v_3, v_4$, and let $v_2=w_1, w_2,\dots, w_n=v_4$ be all the neighbours of $v_1$. Since $\operatorname{deg}_h(v_1)=1$ the edges $w_1w_2,\dots,w_{n-1}w_n$ exist. 
There are two subcases.

\begin{description}
\item[(i)] $v_3\neq w_i$, for every $i\in\{2,\dots, n-1\}$.\\
By Lemma \ref{degh1lemma} we have $\operatorname{deg}(v_1)=4$.
By the uncontractibility of the edges $v_1w_2$ and $v_1w_3$ the edges $w_2w_4$ and $w_1w_3$  must exist. Thus $G$ contains the surface graph  in Figure \ref{f:1_6_alpha}, except possibly for the edge $v_3w_3$.
It follows that the 4-cycle $w_1w_2w_4w_3w_1$ must be the boundary of a 4-hole $H_2$, since otherwise the 5-cycle $v_1w_1w_3w_2w_4v_1$ contains all the holes, in the sense, as before, of being the boundary of an embedded disc in $\P$, $B$ say, which contains the holes. This contradicts $(3,6)$-tightness. We claim now that the edge $v_3w_2$ or $v_3w_3$ must exist, for otherwise there is a contractible edge in $B$. To see this check that since $\operatorname{deg}(v_3)\geq 3$ there exists a vertex $z$ in the interior of the 5-cycle $v_3w_4w_2w_3w_1v_3$, such that $v_3z\in E(G)$. Since $v_3z$ does not lie on a non facial 3-cycle, it follows that it lies on the boundary of the third 4-hole. Thus, if $v_3w_3$ is not allowed, we may assume by symmetry that $w_1z$ is an $FF$ edge in $E(G)$, and so it lies on the non-facial 3 cycle $w_1zw_2w_1$. Hence the third hole is described by the 4-cycle $w_4v_3zw_1w_4$. However, this implies that $zw_3\in E(G)$, which is a contractible $FF$ edge, so we have proved the claim. Hence without loss of generality $G$ contains the surface subgraph $G^1_{6,\alpha}$ as  indicated in Figure \ref{f:1_6_alpha}. Since $G$ is uncontractible it follows that the 3-cycle $v_3w_3w_1$ is a face and so $G=G^1_{6,\alpha}$.



\begin{figure}[h!]
\centering
\begin{tikzpicture}
\path [fill=red3] (0.5,3.7) -- (2,3) -- (2.75,1.701) -- (2.5,1) -- (0.5,1) -- (0.5,3.7)  ;
\path [fill=red3] (6.5, 2.3) -- (4.75,1.701) --(4.75,4.299) --(4.9,5)--(6.5,5)--(6.5,2.3);
\path [fill=red1] (2.75,1.701) -- (3.5,3) -- (2.75,4.299)--(2,3)-- (2.75,1.701);
\path [fill=red2] (2.75,4.299)--(3,5)--(4.9,5)--(4.75,4.299)-- (2.75,4.299);
\path [fill=red2] (4.5,1)--(4.75,1.701) -- (2.75,1.701) -- (2.5,1) -- (4.5,1) ;
\path [fill=green6] (6.5, 2.3) -- (4.75,1.701) --(4.5,1) --(6.5,1)--(6.5,2.3); ;
\path [fill=green6] (0.5,3.7) -- (2,3) -- (2.75,4.299) -- (3,5) -- (0.5,5) -- (0.5,3.7) ;
\path [fill=green1] (2.75,4.299) -- (3.5,3) -- (4.75,4.299) -- (2.75,4.299) ;
\path [fill=green3]  (4.75,4.299) -- (3.5,3) -- (4.75,1.701) -- (4.75,4.299) ;
\path [fill=green4] (4.75,1.701) -- (3.5,3) -- (2.75,1.701) -- (4.75,1.701) ;

\draw [red, thick](2.75,4.299)--(3,5);
\draw [red, thick](4.5,1)--(4.75,1.701);
\draw[green, thick](4.75,4.299)-- (4.9,5);
\draw[green, thick](2.5,1)--(2.75,1.701);
\draw [blue, thick](2,3)--(0.5,3.7);
\draw [blue, thick](6.5, 2.3)--(4.75,1.701);

\draw [->, thick] (0.5,1) -- (4,1);
\draw [thick](4,1) -- (6.5,1);
\draw [<-, thick] (3, 5) -- (6.5,5);
\draw [thick](0.5,5) -- (3,5);
\draw [->, thick] (6.5,1) -- (6.5,3);
\draw [thick](6.5,3) -- (6.5,5);
\draw [<-, thick] (0.5, 2) -- (0.5,5);
\draw [thick](0.5,1) -- (0.5,2);

\draw [thick](2.75,1.701)--(4.75,1.701)--(4.75,4.299)--(2.75,4.299)--(2,3)--(2.75,1.701);
\draw [thick](2.75,1.701) -- (3.5,3) -- (2.75,4.299);
\draw [thick](3.5,3) -- (4.75,1.701);
\draw [thick](3.5,3) -- (4.75,4.299);
\draw [fill] (3.5,3) circle [radius=0.05];
\node [right] at (3.7,2.8) {$v_1$};
\draw [fill] (2.75,1.701) circle [radius=0.05];
\node [below] at (2.75,1.7) {$w_{4}$};
\draw [fill] (4.75,1.701) circle [radius=0.05];
\node [right] at (4.6,1.3) {$w_{3}$};
\draw [fill] (4.75,4.299) circle [radius=0.05];
\node [right] at (4.75,4.3) {$w_{2}$};
\draw [fill] (2.75,4.299) circle [radius=0.05];
\node [left] at (2.7,4.299) {$w_{1}$};
\draw [fill] (2,3) circle [radius=0.05];
\node [left] at (2,3) {$v_{3}$};

\end{tikzpicture}
\caption{The uncontractible surface graph $G^1_{6,\alpha}$.}\label{f:1_6_alpha}
\end{figure}

\item[(ii)]  $v_3=w_{i_0}$ for some $i_0\in \{3,\dots,n-2\}$.\\
By Lemma \ref{degh1lemma}  $\operatorname{deg}(v_1)=5$ and so $v_3=w_3$. Since $v_1w_2$ is an $FF$  edge, it follows that $w_2w_4\in E(G)$ and so $G$ contains the surface graph
$G=G^1_{6,\beta}$ of Figure \ref{f:1_6_beta}.
Since $G$ is uncontractible it follows as before  
that it is equal to $G$.


\begin{figure}[h!]
\centering
\begin{tikzpicture}
\path [fill=red1] (2,3)--(2.75,1.701)--(3.5,3)--(2.75,4.299)--(2,3);
\path [fill=red2] (2,1.701)--(2,3)--(2.75,1.701)--(6.5,1.701)--(6.5,1)--(3.25,1)-- (2,1.701);
\path [fill=red2] (2,4.299)--(3.25,5)--(0.5,5)--(0.5,4.299)--(2,4.299) ;
\path [fill=red3] (2,4.299)--(2,3)--(2.75,4.299)--(6.5,4.299)--(6.5,5)--(3.25,5)-- (2,4.299);
\path [fill=red3] (2,1.701)--(3.25,1)--(0.5,1)--(0.5,1.701)--(2,1.701) ;
\path [fill=green1] (2.75,4.299) -- (3.5,3) -- (6.5,4) -- (6.5,4.299) -- (2.75,4.299);
\path [fill=green2] (2.75,1.701) -- (3.5,3) -- (6.5,2) -- (6.5,1.701) -- (2.75,1.701);
\path [fill=green3] (3.5,3) -- (6.5,3) -- (6.5,4) -- (3.5,3);
\path [fill=green4] (3.5,3) -- (6.5,3) -- (6.5,2) -- (3.5,3);
\path [fill=green1] (0.5,1.701) -- (2,1.701) -- (0.5,2) -- (0.5,1.701);
\path [fill=green2] (0.5,4.299) -- (2,4.299) -- (0.5,4) -- (0.5,4.299);
\path [fill=green3] (0.5,2) -- (2,1.701) -- (2,3) -- (0.5,3)--(0.5,2);
\path [fill=green4] (0.5,4) -- (2,4.299) -- (2,3) -- (0.5,3)--(0.5,4);

\draw [->, thick] (0.5,1) -- (4.5,1);
\draw [thick](4.5,1) -- (6.5,1);
\draw [<-, thick] (2.5, 5) -- (6.5,5);
\draw [thick](0.5,5) -- (2.5,5);
\draw [->, thick] (6.5,1) -- (6.5,2.5);
\draw [thick](6.5,2.5) -- (6.5,5);
\draw [<-, thick] (0.5, 2.5) -- (0.5,5);
\draw [thick](0.5,1) -- (0.5,2.5);

\draw [thick](2,3)--(2.75,1.701)--(3.5,3)--(2.75,4.299)--(2,3);
\draw [thick](2.75,1.701)--(6.5,1.701);
\draw [thick](2.75,4.299)--(6.5,4.299);
\draw [thick](2,3)--(2,4.299)--(0.5,4.299);
\draw [thick](2,3)--(2,1.701)--(0.5,1.701);
\draw [thick](6.5,2) -- (3.5,3) -- (6.5,4);
\draw [thick](3.5,3) -- (6.5,3);
\draw [thick](0.5,4) -- (2,4.299);
\draw [thick](0.5,2) -- (2,1.701);
\draw [thick](0.5,3) -- (2,3);
\draw [blue, thick](3.25,5) -- (2,4.299);
\draw [blue, thick](3.25,1) -- (2,1.701);

\draw [fill] (3.5,3) circle [radius=0.05];
\node [below] at (3.7,2.8) {$v_1$};
\draw [fill] (2.75,1.701) circle [radius=0.05];
\node [below] at (2.75,1.7) {$v_{4}$};
\draw [fill] (2,1.701) circle [radius=0.05];
\node [above left] at (2,1.701) {$w_{4}$};
\draw [fill] (2,4.299) circle [radius=0.05];
\node [below left] at (2,4.299) {$w_{2}$};
\draw [fill] (2.75,4.299) circle [radius=0.05];
\node [above right] at (2.75,4.299) {$v_{2}$};
\draw [fill] (2,3) circle [radius=0.05];
\node [right] at (2,3) {$v_{3}$};

\end{tikzpicture}
\caption{The uncontractible surface graph $G^1_{6,\beta}$.}\label{f:1_6_beta}
\end{figure}

\end{description}

{\bf Case (b).} 
Let $v_1$ lie on the boundary of a 5-hole $H$ with boundary edges $v_1v_2$, $v_2v_3$, $v_3v_4$, $v_4v_5$, $v_5v_1$. We may assume that $\operatorname{deg}_h(v_i)=2$, for every $i=2,3,4,5$, since otherwise there is a vertex $v$ on a 4-hole of $G$.
Since $G$ has two holes it is straightforward to check that $\operatorname{deg}(v_1)=4$ and that the second hole is described by the 4-cycle $v_2v_3v_5v_4v_2$. Thus we obtain that $G$ is the uncontractible (3,6)-tight graph $G_5^1$ given by Figure \ref{f:1_5}.

\begin{figure}[h!]
\centering
\begin{tikzpicture}
\path [fill=red1] (2,3)--(2.75,1.701)--(4,1.701)--(4,4.299)--(2.75,4.299)--(2,3);
\path [fill=red2] (0.5,1.701)--(6.5,1.701)--(6.5,1)--(0.5,1)-- (0.5,1.701);
\path [fill=red2] (0.5,4.299)--(6.5,4.299)--(6.5,5)--(0.5,5)--(0.5,4.299) ;
\path [fill=green1] (4,4.299) -- (6.5,4.299) -- (6.5,3.5) -- (4,4.299);
\path [fill=green2] (4,4.299) -- (6.5,3.5) -- (6.5,2.5) -- (4,1.701) -- (4,4.299);
\path [fill=green3] (4,1.701) -- (6.5,1.701) -- (6.5,2.5) -- (4,1.701);;
\path [fill=green1] (2.75,1.701) -- (0.5,1.701) -- (0.5,2.5)--(2,3) -- (2.75,1.701);
\path [fill=green2] (2,3) -- (0.5,2.5) -- (0.5,3.5) -- (2,3);
\path [fill=green3] (2.75,4.299) -- (0.5,4.299) -- (0.5,3.5)--(2,3) -- (2.75,4.299);

\draw [->, thick] (0.5,1) -- (4.5,1);
\draw [thick](4.5,1) -- (6.5,1);
\draw [<-, thick] (2.5, 5) -- (6.5,5);
\draw [thick](0.5,5) -- (2.5,5);
\draw [->, thick] (6.5,1) -- (6.5,2.5);
\draw [thick](6.5,2.5) -- (6.5,5);
\draw [<-, thick] (0.5, 2.5) -- (0.5,5);
\draw [thick](0.5,1) -- (0.5,2.5);

\draw [thick](2,3)--(2.75,1.701)--(4,1.701)--(4,4.299)--(2.75,4.299)--(2,3);
\draw [thick](0.5,1.701)--(6.5,1.701);
\draw [thick](0.5,4.299)--(6.5,4.299);
\draw [thick](0.5,2.5) -- (2,3);
\draw [thick](0.5,3.5) -- (2,3);
\draw [thick](4,4.299) -- (6.5,3.5);
\draw [thick](4,1.701) -- (6.5,2.5);

\draw [fill] (2,3) circle [radius=0.05];
\node [right] at (2,3) {$v_1$};
\draw [fill] (2.75,1.701) circle [radius=0.05];
\node [above right] at (2.75,1.7) {$v_{5}$};
\draw [fill] (4,1.701) circle [radius=0.05];
\node [above left] at (4,1.701) {$v_{4}$};
\draw [fill] (4,4.299) circle [radius=0.05];
\node [below left] at (4,4.299) {$v_{3}$};
\draw [fill] (2.75,4.299) circle [radius=0.05];
\node [below right] at (2.75,4.299) {$v_{2}$};

\end{tikzpicture}
\caption{The uncontractible surface graph $G^1_5$.}\label{f:1_5}
\end{figure}
\end{proof}



\begin{prop}
Let $G\in \fP_k$, for $k=1, 2$ or 3, be uncontractible with
$\operatorname{deg}_h(v)\geq 2$ for all $v\in V(G)$.  Then $G$ is one of the four surface graphs $G^2_{6,\alpha}, G^2_{6,\beta}, G^3_4, G_3^2$.
\end{prop}
\begin{proof} 
Suppose first that
$G$ has 2 or 3 holes. Then the hole boundaries have length 4 or 5 and it follows from the simplicity of $G$ that every vertex is common to at least 2 holes. Since there are either 2 or 3 holes it follows readily that $|V|\leq 6$.

{\bf Case (a).} Suppose that $G$ contains at least one $FF$  edge, say $v_1v_2$, with non facial  3-cycle $v_1v_2v_3$, and associated 3-cycle faces  $v_1v_2v_4v_1$ and $v_1v_2v_5v_1$.
We claim that one of the edges $v_3v_4$ or $v_3v_5$ lies in $E(G)$. Suppose, by way of contradiction, that neither edge exists. Then we show that the edge $v_4v_5$ is also absent. Indeed, if $v_4v_5\in E(G)$, then we have two planar 5-cycles; $v_1v_4v_5v_2v_3v_1$ and $v_1v_5v_4v_2v_3v_1$, as in Figure \ref{f:two5cycles}. 
\begin{center}
\begin{figure}[ht]
\centering
\includegraphics[width=4.5cm]{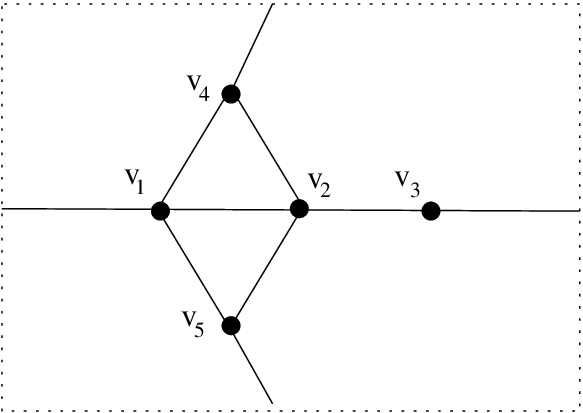}
\caption{A subgraph with the 5-cycles $v_1v_4v_5v_2v_3v_1$ and $v_1v_5v_4v_2v_3v_1$.}
\label{f:two5cycles}
\end{figure}
\end{center}
By the sparsity condition one of these has a vertex in the interior with 3 incident edges and the other has a single chordal edge in the interior and by symmetry we may assume that the planar 5-cycle $v_1v_4v_5v_2v_3v_1$  has the single chordal edge. However, of the 5 possibilities $v_1v_2, v_2v_4, v_1v_5$ are not available, by the simplicity of $G$, and the edges $v_3v_4, v_3v_5$ are absent by assumption. This contradiction shows that $v_4v_5$ is indeed absent and so, since $v_4, v_5$ have degree at least 2, the edges $v_4v_6, v_5v_6$ must exist. Now the complement of the 2 3-cycle faces is bounded by two 6-cycles. By the sparsity condition there are now only 2 further edges to add and so there must be a 5-cycle hole, a contradiction, and so the claim holds.

Without loss of generality we suppose that $v_3v_4\in E(G)$. Since $\operatorname{deg}_h(v_2)\geq 2$, it follows that
$v_2v_6\in E(G)$. Moreover, the edges $v_6v_2$,$v_2v_3$ should be on the boundary of a planar 4-hole $H_1$, and this implies that $v_1v_6\in E(G)$. Similarly we obtain that the two remaining holes are determined by the cycles $v_1v_3v_4v_5v_1$, and $v_2v_5v_4v_6v_2$. The resulting (3,6)-tight triangulated surface graph is given in Figure \ref{f:CaseAholedegree2} and is the uncontractible surface graph $G_{6,\alpha}^2$.

\begin{figure}[h!]
\centering
\begin{tikzpicture}
\path [fill=red1] (0.5,1)--(0.5,1.35)--(2,1.7)--(3.5,1.7)--(3.5,1)--(0.5,1);
\path [fill=red1] (3,5,5)--(3.5,4.3)--(4.25,3)--(6.5,4.65)--(6.5,5)--(3.5,5);
\path [fill=red2] (0.5,5)--(3.5,5)--(3.5,4.3)--(2.75,3)-- (0.5,4.65)--(0.5,5);
\path [fill=red2] (3.5,1)--(3.5,1.7)--(5,1.7)--(6.5,1.35)--(6.5,1)--(3.5,1);
\path [fill=red3] (0.5,1.35)--(2,1.7)--(2.75,3)--(0.5,4.65)--(0.5,1.35);
\path [fill=red3] (6.5,1.35)--(5,1.7)--(4.25,3)--(6.5,4.65)--(6.5,1.35);
\path [fill=green1] (2,1.7) -- (3.5,1.7) -- (2.75,3) -- (2,1.7);
\path [fill=green2] (3.5,1.7) -- (2.75,3)--(4.25,3)--(3.5,1.7);
\path [fill=green3] (4.25,3)--(3.5,1.7)--(5,1.7)--(4.25,3);
\path [fill=green4] (2.75,3)--(4.25,3)--(3.5,4.3)--(2.75,3);

\draw [->, thick] (0.5,1) -- (4.5,1);
\draw [thick](4.5,1) -- (6.5,1);
\draw [<-, thick] (2.5, 5) -- (6.5,5);
\draw [thick](0.5,5) -- (2.5,5);
\draw [->, thick] (6.5,1) -- (6.5,2.5);
\draw [thick](6.5,2.5) -- (6.5,5);
\draw [<-, thick] (0.5, 2.5) -- (0.5,5);
\draw [thick](0.5,1) -- (0.5,2.5);

\draw [thick](2,1.7)--(3.5,1.7)--(5,1.7);
\draw [thick](2,1.7)--(2.75,3)--(3.5,1.7)--(4.25, 3)--(5,1.7);
\draw [thick](2.75,3)--(4.25,3);
\draw [thick](2.75,3)--(3.5,4.3)--(4.25,3);
\draw [green, thick](5,1.7)--(6.5,1.35);
\draw [green, thick](0.5,4.65)--(2.75,3);
\draw [red, thick](0.5,1.35)--(2,1.7);
\draw [red, thick](4.25,3)--(6.5,4.65);
\draw [blue, thick](3.5,1)--(3.5,1.7);
\draw [blue, thick](3.5,4.3)--(3.5,5);
\draw [fill] (2,1.7) circle [radius=0.05];
\node [above left] at (2,1.7) {$v_5$};
\draw [fill] (3.5,1.7) circle [radius=0.05];
\node [below right] at (3.5,1.7) {$v_{2}$};
\draw [fill] (5,1.7) circle [radius=0.05];
\node [above right] at (5,1.7) {$v_{3}$};
\draw [fill] (2.75,3) circle [radius=0.05];
\node [left] at (2.75,3) {$v_{1}$};
\draw [fill] (4.25,3) circle [radius=0.05];
\node [right] at (4.25,3) {$v_{4}$};
\draw [thick] (3.5,4.3) circle [radius=0.05];
\node [above right] at (3.5,4.3) {$v_{6}$};

\end{tikzpicture}
\caption{The uncontractible surface graph with $h(G)=2$ and an $FF$ edge; $G_{6,\alpha}^2$.}\label{f:CaseAholedegree2}
\end{figure}

{\bf Case (b).} Suppose now $G$ has at least one 3-cycle face, $v_1v_2v_3$, and no $FF$  edges. 
Then the edge $v_1v_2$ is on the boundary of a 4-hole $H_1$, that is determined by the edges  
$v_1v_2$, $v_2v_4$, $v_4v_5$ and $v_5v_1$.  

To see that $|V|\neq 5$ note that without loss of generality the edge $v_3v_4$ exists and $G$ contains the surface subgraph shown in Figure \ref{f:case_b_noFF}. Also, since $v_5$ cannot have degree 2 at least one of the edges $v_5v_3, v_5v_2$ exists.
\begin{center}
\begin{figure}[ht]
\centering
\includegraphics[width=4.5cm]{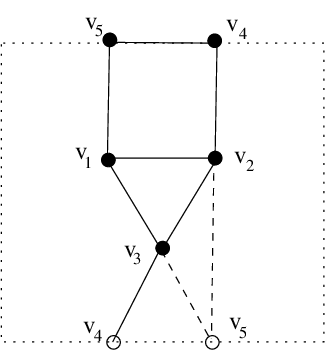}
\caption{A necessary subgraph.}
\label{f:case_b_noFF}
\end{figure}
\end{center}
If $v_5v_2$ exists then the edge $v_2v_3$ is adjacent to a 4-cycle hole and $v_5v_3$ is absent. We note next that the planar 5-cycle $v_3v_1v_5v_2v_4v_3$
must contain a chord edge (and so provide the third 4-cycle hole). The only available edge (by simplicity) is $v_3v_5$. This however is inadmissible since it introduces a second 3-cycle face $v_3v_5v_1$ adjacent to $v_1v_2v_3$.

Similarly, if $v_5v_3$ exists then we have the planar 6-cycle
$v_3v_1v_5v_3v_2v_4v_3$ and there must exist a diameter edge to create the 2 additional 4-cycle holes. As there is no such edge we conclude that $|V|=6$.

Introducing $v_6$ the fact that $v_2v_3$ and $v_3v_1$ lie on 4-cycle hole boundaries leads to the surface  graph $G_{6,\beta}^2$ indicated in Figure \ref{f:CaseBholedegree2}. 

\begin{figure}[h!]
\centering
\begin{tikzpicture}
\path [fill=red1] (2.5,4)--(2.5,5)--(4.5,5)--(4.5,4)--(2.5,4);
\path [fill=red2] (2.5,4)--(0.5,3)--(0.5,1)--(2.5,1)-- (3.5,2)--(2.5,4);
\path [fill=red3] (4.5,4)--(6.5,3)--(6.5,1)--(4.5,1)-- (3.5,2)--(4.5,4);
\path [fill=green1] (2.5,4) -- (4.5,4) -- (3.5,2) -- (2.5,4);
\path [fill=green2] (2.5,4)--(0.5,3)--(0.5,5)--(2.5,5)--(2.5,4);
\path [fill=green3] (4.5,4)--(6.5,3)--(6.5,5)--(4.5,5)--(4.5,4);
\path [fill=green4] (2.5,1)--(4.5,1)--(3.5,2)--(2.5,1);

\draw [->, thick] (0.5,1) -- (5,1);
\draw [thick](5,1) -- (6.5,1);
\draw [<-, thick] (1.5, 5) -- (6.5,5);
\draw [thick](0.5,5) -- (1.5,5);
\draw [->, thick] (6.5,1) -- (6.5,2.5);
\draw [thick](6.5,2.5) -- (6.5,5);
\draw [<-, thick] (0.5, 2.5) -- (0.5,5);
\draw [thick](0.5,1) -- (0.5,2.5);

\draw [thick](2.5,4)--(3.5,2)--(4.5,4)--(2.5,4);
\draw [thick](2.5,4)--(2.5,5);
\draw [thick](2.5,4)--(0.5,3);
\draw [thick](4.5,4)--(4.5,5);
\draw [thick](4.5,4)--(6.5,3);
\draw [thick](3.5,2)--(2.5,1);
\draw [thick](3.5,2)--(4.5,1);
\draw [green, thick](2.5,1)--(4.5,1);
\draw [green, thick](2.5,5)--(4.5,5);
\draw [red, thick](0.5,3)--(0.5,1)--(2.5,1);
\draw [red, thick](4.5,5)--(6.5,5)--(6.5,3);
\draw [blue, thick](6.5,3)--(6.5,1)--(4.5,1);
\draw [blue, thick](2.5,5)--(0.5,5)--(0.5,3);
\draw [fill] (2.5,4) circle [radius=0.05];
\node [above left] at (2.5,4) {$v_1$};
\draw [fill] (4.5,4) circle [radius=0.05];
\node [above right] at (4.5,4) {$v_{2}$};
\draw [fill] (3.5,2) circle [radius=0.05];
\node [below] at (3.5,2) {$v_{3}$};
\draw [fill] (4.5,5) circle [radius=0.05];
\node [below left] at (4.5,5) {$v_{4}$};
\draw [fill] (2.5,5) circle [radius=0.05];
\node [below right] at (2.5,5) {$v_{5}$};
\draw [thick] (0.5,3) circle [radius=0.05];
\node [above right] at (0.5,3.2) {$v_{6}$};
\draw [thick] (6.5,3) circle [radius=0.05];
\node [below left] at (6.5,3) {$v_{6}$};
\draw [thick] (2.5,1) circle [radius=0.05];
\node [above left] at (2.5,1) {$v_{4}$};
\draw [thick] (4.5,1) circle [radius=0.05];
\node [above right] at (4.5,1) {$v_{5}$};

\end{tikzpicture}
\caption{The uncontractible surface graph $G_{6,\beta}^2$, with $h(G)=2$, no $FF$ edge and a 3-cycle face.}\label{f:CaseBholedegree2}
\end{figure}

{\bf Case (c).} Let now $G$ be a surface graph with no 3-cycle faces. Since $\operatorname{deg}(v)\geq 3$ for each vertex it follows that $\operatorname{deg}_h(v)= 3$ and $\deg(v)=3$, for all $v\in V(G)$. Thus
$|V|=4$ and it follows that $G$ is the uncontractible (3,6)-tight surface graph $G_4^3$ given by Figure \ref{f:coloured4vertexgraph}. 

\begin{figure}[h!]
\centering
\begin{tikzpicture}
\path [fill=red1] (0.5,1)--(3.5,3)--(3.5,5)--(0.5,5)--(0.5,1);
\path [fill=red2] (6.5,1)--(3.5,3)--(3.5,5)--(6.5,5)--(6.5,1);
\path [fill=red3] (0.5,1)--(6.5,1)--(3.5,3)--(0.5,1);

\draw [->, thick] (0.5,1) -- (5,1);
\draw [thick](5,1) -- (6.5,1);
\draw [<-, thick] (1.5, 5) -- (6.5,5);
\draw [thick](0.5,5) -- (1.5,5);
\draw [->, thick] (6.5,1) -- (6.5,2.5);
\draw [thick](6.5,2.5) -- (6.5,5);
\draw [<-, thick] (0.5, 2.5) -- (0.5,5);
\draw [thick](0.5,1) -- (0.5,2.5);

\draw [thick](0.5,1)--(3.5,3)--(3.5,5);
\draw [thick](3.5,3)--(6.5,1);
\draw [green, thick](0.5,1)--(3.5,1);
\draw [green, thick](3.5,5)--(6.5,5);
\draw [red, thick](3.5,1)--(6.5,1);
\draw [red, thick](0.5,5)--(3.5,5);
\draw [blue, thick](0.5,1)--(0.5,5);
\draw [blue, thick](6.5,1)--(6.5,5);
\draw [fill] (3.5,3) circle [radius=0.05];
\node [above left] at (3.5,3) {$v_1$};
\draw [fill] (3.5,5) circle [radius=0.05];
\node [below left] at (3.5,5) {$v_{2}$};
\draw [fill] (6.5,5) circle [radius=0.05];
\node [below left] at (6.5,5) {$v_{3}$};
\draw [fill] (6.5,1) circle [radius=0.05];
\node [above left] at (6.2,0.9) {$v_{4}$};
\draw [thick] (3.5,1) circle [radius=0.05];
\node [above right] at (3.5,1) {$v_{2}$};
\draw [thick] (0.5,1) circle [radius=0.05];
\node [above right] at (0.4,1.2) {$v_{3}$};
\draw [thick] (0.5,5) circle [radius=0.05];
\node [below right] at (0.5,5) {$v_{4}$};
\end{tikzpicture}
\caption{The uncontractible surface graph  $G_4^3$ with $h(G)=3$.}\label{f:coloured4vertexgraph}
\end{figure}

{\bf Case (d).}
Finally, suppose that $G\in \mathfrak{P}_1$. We claim that the $G$  has no faces and the surface graph is given by 
Figure \ref{f:G3,2}.

Assume first that there exists an $FF$ edge, say $v_1v_2$, that lies on the faces $v_1v_2v_3v_1$ and $v_1v_2v_4v_1$. By the uncontractibility,  $v_1v_2$ lies on a non facial 3-cycle $v_1v_2v_5v_1$. Note that $v_3v_4\notin E(G)$, since otherwise the 6-hole would lie inside a 5-cycle, either $v_1v_3v_4v_2v_5v_1$ or $v_1v_4v_3v_2v_5v_1$, contradicting the sparsity requirement. It follows that we cannot have $|V(G)|\leq 5$.
Indeed, in this case (see Figure \ref{f:lessthan5}) $v_3v_5\in E(G)$, since $\operatorname{deg}(v_3)\geq 3$, and so without loss of generality, in view of the symmetry,  $v_1v_3$ is an $FF$ edge. But this edge does not lie on a non-facial 3-cycle, a contradiction. 

\begin{figure}[h!]
\centering
\begin{tikzpicture}

\draw [->, thick] (0.5,1) -- (4.5,1);
\draw [thick](4.5,1) -- (6.5,1);
\draw [<-, thick] (2.5, 5) -- (3.45,5);
\draw [thick] (3.55,5)--(6.5,5);
\draw [thick](0.5,5) -- (2.5,5);
\draw [->, thick] (6.5,1) -- (6.5,2.5);
\draw [thick](6.5,2.5) -- (6.5,5);
\draw [<-, thick] (0.5, 2.5) -- (0.5,5);
\draw [thick](0.5,1) -- (0.5,2.5);

\draw [thick](3.5,2)--(3.5,4);
\draw [thick](3.5,2)--(2.5,3)--(3.5,4);
\draw [thick](3.5,2)--(4.5,3)--(3.5,4);
\draw [thick](3.5,1)--(3.5,2);
\draw [thick](3.5,4)--(3.5,4.95);

\draw [fill] (3.5,4) circle [radius=0.05];
\node [above right] at (3.5,4) {$v_1$};
\draw [fill] (3.5,2) circle [radius=0.05];
\node [below right] at (3.5,2) {$v_{2}$};
\draw [fill] (2.5,3) circle [radius=0.05];
\node [right] at (2.5,3) {$v_{3}$};
\draw [fill] (4.5,3) circle [radius=0.05];
\node [right] at (4.5,3) {$v_{4}$};
\draw [fill] (3.5,1) circle [radius=0.05];
\node [above left] at (3.5,1) {$v_{5}$};
\draw (3.5,5) circle [radius=0.05];
\node [below left] at (3.5,5) {$v_{5}$};
\end{tikzpicture}
\caption{$|V(G)|\leq 5$ leads to a contradiction.}\label{f:lessthan5}
\end{figure}

Thus $|V(G)|=6$ and it remains to consider two subcases:

\begin{enumerate}
\item[(i)] $v_3v_5\in E(G)$. In this case $v_1v_3$ lies on the non-facial 3-cycle $v_1v_3v_6v_1$. However, this leads to a contradiction, since the 6-hole is contained either in the 5-cycle $v_5v_3v_6v_1v_2v_5$ or in the 5-cycle $v_6v_3v_2v_5v_1v_6$. Hence by symmetry neither of the edges $v_3v_5,v_4v_5$ is allowed.
\item[(ii)] $v_3v_6,v_4v_6\in E(G)$. In this case, indicated in Figure \ref{f:2edges},  we may assume that the hole is contained in the planar 6-cycle $v_1v_5v_2v_4v_6v_3v_1$ and that the planar 6-cycle $v_1v_5v_2v_3v_6v_4v_1$ is triangulated. This implies that $v_2v_3$ is an $FF$ edge and so lies on non-facial 3-cycle. However, the only candidate cycle is $v_3v_2v_6v_3$ and if $v_2v_6$ lies in $E(G)$ then the hole is contained in the 5-cycle $v_1v_5v_2v_6v_3v_1$, a contradiction.

\begin{figure}[h!]
\centering
\begin{tikzpicture}

\draw [->, thick] (0.5,1) -- (4.5,1);
\draw [thick](4.5,1) -- (6.5,1);
\draw [<-, thick] (2.5, 5) -- (3.45,5);
\draw [thick] (3.55,5)--(6.5,5);
\draw [thick](0.5,5) -- (2.5,5);
\draw [->, thick] (6.5,1) -- (6.5,2.5);
\draw [thick](6.5,2.5) -- (6.5,5);
\draw [<-, thick] (0.5, 2.5) -- (0.5,5);
\draw [thick](0.5,1.65) -- (0.5,2.5);
\draw [thick](0.5,1) -- (0.5,1.55);

\draw [thick](3.5,2)--(3.5,4);
\draw [thick](3.5,2)--(2.5,3)--(3.5,4);
\draw [thick](3.5,2)--(4.5,3)--(3.5,4);
\draw [thick](3.5,1)--(3.5,2);
\draw [thick](3.5,4)--(3.5,4.95);
\draw [thick](2.5,3)--(0.53,1.63);
\draw [thick](4.5,3)--(6.47,4.37);

\draw [fill] (3.5,4) circle [radius=0.05];
\node [above right] at (3.5,4) {$v_1$};
\draw [fill] (3.5,2) circle [radius=0.05];
\node [below right] at (3.5,2) {$v_{2}$};
\draw [fill] (2.5,3) circle [radius=0.05];
\node [right] at (2.5,3) {$v_{3}$};
\draw [fill] (4.5,3) circle [radius=0.05];
\node [right] at (4.5,3) {$v_{4}$};
\draw [fill] (3.5,1) circle [radius=0.05];
\node [above left] at (3.5,1) {$v_{5}$};
\draw (3.5,5) circle [radius=0.05];
\node [below left] at (3.5,5) {$v_{5}$};
\draw (0.5,1.6) circle [radius=0.05];
\node [below right] at (0.5,1.6) {$v_{6}$};
\draw [fill] (6.5,4.4) circle [radius=0.05];
\node [above left] at (6.5,4.4) {$v_6$};
\end{tikzpicture}
\caption{Edges $v_3v_6, v_4v_6$ in $G$ leads to a contradiction.}\label{f:2edges}
\end{figure}

We have shown that no $FF$ edge is allowed. Suppose now that $G$ contains a face, described by the vertices  $v_1,v_2$ and $v_3$. Since there are no $FF$ edges, all edges $v_1v_2,v_2v_3$ and $v_1v_3$ lie on the boundary of the hole. Moreover, since they form a face of the surface graph, they cannot form a 3-cycle path in the boundary of the hole. Only 3 edges of the boundary cycle are left to be determined, so we may assume that the path $v_1v_2v_3$ lies on the boundary. Therefore, without loss of generality, there exists a vertex $v_4$ on the boundary that connects the two paths, $v_1v_2v_3$and $v_1v_3$, so we obtain the 5-path  $v_1v_3v_4v_1v_2v_3$. But this implies that the remaining edge of the 6-hole is $v_1v_3$, which would contradict graph simplicity. Hence the surface graph contains no faces and the proof is complete.
\end{enumerate}
\end{proof}

\begin{figure}[h!]
\centering
\begin{tikzpicture}
\path [fill=red1] (0.5,1)--(6.5,1)--(6.5,5)--(0.5,5)--(0.5,1);

\draw [->, thick] (0.5,1) -- (4.5,1);
\draw [thick](4.5,1) -- (6.5,1);
\draw [<-, thick] (2.5, 5) -- (6.5,5);
\draw [thick](0.5,5) -- (2.5,5);
\draw [->, thick] (6.5,1) -- (6.5,2.5);
\draw [thick](6.5,2.5) -- (6.5,5);
\draw [<-, thick] (0.5, 2.5) -- (0.5,5);
\draw [thick](0.5,1) -- (0.5,2.5);

\draw [thick](3.5,1.7)--(3.5,3)--(3.5,4.3);
\draw [blue, thick](3.5,1)--(3.5,1.7);
\draw [blue, thick](3.5,4.3)--(3.5,5);

\draw [fill] (3.5,1.7) circle [radius=0.05];
\node [right] at (3.5,1.7) {$v_1$};
\draw [fill] (3.5,3) circle [radius=0.05];
\node [right] at (3.5,3) {$v_{2}$};
\draw [fill] (3.5,4.3) circle [radius=0.05];
\node [right] at (3.5,4.3) {$v_{3}$};
\end{tikzpicture}
\caption{The uncontractible surface graph $G^2_3$.}\label{f:G3,2}
\end{figure}

\end{document}